\documentclass[manyauthors]{fundam}

\publyear{22}
\papernumber{2102}
\volume{185}
\issue{1}

\usepackage[T1]{fontenc}

\usepackage{graphicx,graphbox}
\usepackage[all]{xy}

\usepackage{amsmath,amssymb,mathtools}
\usepackage{booktabs}
\usepackage{mathpartir}
\usepackage{tikz-cd}
\usepackage{enumitem}

\usepackage{pifont}

\usepackage{bbding}

\newcommand{\application}{\bullet} 

\newcommand{\ip}{\rotatebox[origin=c]{180}{$\pi$}}

\begin{document}

\title{A Structural Account of Combinatory Completeness}
\author{
  Ivan Kuzmin\thanks{Ivan Kuzmin and Chad Nester were supported by the Estonian Research Council grant PRG2764.} \\ University of Tartu \\ ivan.kuzmin@ut.ee
\and
  Chad Nester\thanksas{1} \\ University of Tartu \\ nester@ut.ee
\and
  {\"{U}}lo Reimaa\thanks{Ülo Reimaa was supported by the Estonian Research Council grant PRG1204.} \\ University of Tartu \\ ulo.reimaa@ut.ee
\and
  Sam Speight\thanks{Sam Speight was supported by the Engineering and Physical Sciences Research Council grant EP/W034514/1.} \\ University of Birmingham \\ s.l.speight@bham.ac.uk
}
\address{Chad Nester, Institute of Computer Science, University of Tartu, Narva mnt 18, 51009, Tartu, Estonia}

\runninghead{I. Kuzmin, C. Nester, \"{U}. Reimaa, S. Speight}{A Structural Account of Combinatory Completeness}

\maketitle

\begin{abstract}
  We give a general notion of combinatory completeness with respect to a faithful cartesian club and use it systematically to obtain characterisations of a number of different kinds of applicative system. Moreover, we characterise combinatory completeness of a given applicative system in terms of multicategory structure on its computable maps.

\end{abstract}

\section{Introduction}\label{sec:intro}

Combinatory logic was introduced, essentially independently, by Sch\"{o}nfinkel and Curry~\cite{Schonfinkel1924,Curry1930}, and has played a fundamental role in the development of logic and computer science during the century that followed (see e.g.,~\cite{Cardone2006}).

A \emph{combinatory algebra} is an algebraic model of combinatory logic~\cite{Barendregt1981}. One defines an \emph{applicative system} to be a set $A$ of \emph{combinators} together with a (total) function $\application: A \times A \to A$ called \emph{application}. Adopting the usual conventions, we treat application as a left-associative infix binary operation whose symbol is often omitted, so that $xy = x \application y = \application(x,y)$ and $xyz = (xy)z$. Next, we define:
\begin{itemize}
\item a $\mathsf{B}$ combinator to be some $\mathsf{B} \in A$ such that $\mathsf{B}xyz = x(yz)$ for all $x,y,z \in A$.
\item a $\mathsf{C}$ combinator to be some $\mathsf{C} \in A$ such that $\mathsf{C}xyz = xzy$ for all $x,y,z \in A$.
\item a $\mathsf{K}$ combinator to be some $\mathsf{K} \in A$ such that $\mathsf{K}xy = x$ for all $x,y \in A$.
\item a $\mathsf{W}$ combinator to be some $\mathsf{W} \in A$ such that $\mathsf{W}xy = xyy$ for all $x,y \in A$.
\item an $\mathsf{I}$ combinator to be some $\mathsf{I} \in A$ such that $\mathsf{I}x = x$ for all $x \in A$.
\end{itemize}
We call applicative systems which have some subset $H$ of these combinators $H$-algebras, so that for example a $\mathsf{BI}$-algebra is an applicative system with a $\mathsf{B}$ and $\mathsf{I}$ combinator, and a $\mathsf{BCI}$-algebra is a $\mathsf{BI}$-algebra with a $\mathsf{C}$ combinator. Then in particular a combinatory algebra is defined to be a $\mathsf{BCKWI}$-algebra\footnote{It is of course possible to define combinatory algebras using fewer combinators, and in fact the $\mathsf{I}$ combinator is redundant in our presentation since it is obtainable from $\mathsf{K}$ and $\mathsf{W}$. We feel that this larger combinator basis is better suited to our aims herein.}.

Given an applicative system $(A,\application)$ one defines a \emph{polynomial} in variables $x_1,\ldots,x_n$ to be one of: a variable $x_i$ where $1 \leq i \leq n$; an element $a \in A$; or $t \application s$ where $t$ and $s$ are polynomials in $x_1,\ldots,x_n$. A polynomial $t$ is said to be \emph{computable} in case there exists $a \in A$ such that for all $b_1,\ldots,b_n \in A$ we have:
\[
   ab_1 \cdots b_n = t[b_1,\ldots,b_n/x_1,\ldots,x_n]
\]
For example, if the applicative system in question has a $\mathsf{W}$ combinator, then the polynomial $x_1x_2x_2$ is computable. An applicative system in which every polynomial is computable is said to be \emph{combinatory complete}. This turns out to be equivalent to being a combinatory algebra, as in:
\begin{theorem}[After \cite{Curry1958}, Chapter 6]\label{thm:classical-combinatory-complete}
  An applicative system $(A,\application)$ is combinatory complete if and only if it is a combinatory algebra.
\end{theorem}

In fact, to obtain a combinatory algebra it is enough to ask that only the \emph{regular} polynomials are computable, where a polynomial is said to be regular in case it contains only variables. For example, if $a \in A$ then $x_1(x_2x_3)$ and $x_2a$ are both polynomials in variables $x_1,x_2,x_3$, but the former is regular while the latter is not. For many classes of polynomial, computability of the regular polynomials implies computability of all polynomials, so that they express the same notion of combinatory completeness. However, this is not true for all such classes. We mention this here because the distinction will appear in our development.

A natural question is whether there exist analogues of Theorem~\ref{thm:classical-combinatory-complete} characterising applicative systems in which only some subset of the distinguished elements of a combinatory algebra need exist. For example, it is known that an applicative system is a $\mathsf{BCI}$-algebra if and only if every \emph{linear} polynomial is computable, where a polynomial $t$ in variables $x_1,\ldots,x_n$ is said to be linear in case each variable $x_1,\ldots,x_n$ occurs exactly once in $t$ (see e.g.,~\cite{Simpson05,Hoshino2007}). As far as we are aware, no satisfying answer to the wider question exists in the literature.

The first contribution of the present paper is an answer to this question. We give a general notion of combinatory completeness and use it to obtain a number of combinatory completeness results in a systematic fashion. Specifically, we obtain combinatory completeness results characterising applicative systems with $\mathsf{B}$ and $\mathsf{I}$ combinators together with most subsets of the combinators $\mathsf{C}$,$\mathsf{K}$, and $\mathsf{W}$. 

Central to our approach is the notion of faithful cartesian club~\cite{Shulman2016}, which is a sort of well-behaved subcategory of the category $\mathbf{Fun}$ of functions between sets $\underline{n} = \{1,\ldots,n\}$ (i.e., the skeleton of the category of finite sets and functions). Every faithful cartesian club $\mathfrak{S}$ determines a notion of structured multicategory whose instances are called $\mathfrak{S}$-multicategories. We work with a more abstract notion of applicative system, in which the carrier $A$ becomes an object of some ambient $\mathfrak{S}$-multicategory $\mathcal{M}$ and application becomes a morphism $\application \in \mathcal{M}(A,A;A)$.

The morphisms of the smallest sub-$\mathfrak{S}$-multicategory of $\mathcal{M}$ containing the application morphism play the role of regular polynomials, and we say that an applicative system is \emph{weakly} $\mathfrak{S}$-combinatory complete when every such morphism is computable in an appropriate sense. Similarly, the morphisms of the smallest sub-$\mathfrak{S}$-multicategory of $\mathcal{M}$ containing the application morphism and all of the generalised elements $a \in \mathcal{M}(;A)$ of the carrier play the role of the full collection of polynomials, and when every such morphism is computable we say that the applicative system in question is $\mathfrak{S}$-combinatory complete.

The definition of $\mathsf{B},\mathsf{C},\mathsf{K},\mathsf{W}$, and $\mathsf{I}$ combinators in an applicative system is easily adapted to applicative systems in $\mathfrak{S}$-multicategories, with the caveat that for certain combinators to be expressible the faithful cartesian club $\mathfrak{S}$ must contain certain functions. Specifically, while the definitions of the $\mathsf{B}$ and $\mathsf{I}$ combinator make sense in an $\mathfrak{S}$-multicategory for any $\mathfrak{S}$, to express the $\mathsf{C}$ combinator we require that $\mathfrak{S}$ contains the bijections, with the $\mathsf{K}$ and $\mathsf{W}$ combinators requiring $\mathfrak{S}$ to contain the monotone injections and monotone surjections, respectively. 

We obtain a number of results (Theorem~\ref{thm:id-combinatory-complete} and Theorem~\ref{thm:structured-combinatory-complete}) relating weak $\mathfrak{S}$-combinatory completeness to the existence of certain combinators, summarised in Figure~\ref{fig:results-summary}. The entries of the first column indicate subcategories of $\mathbf{Fun}$ that form faithful cartesian clubs $\mathfrak{S}$, which we will usually refer to by the short names given in the second column. For example, $\mathbf{Inj}$ is the wide subcategory of $\mathbf{Fun}$ that contains only injective functions as morphisms, and $\mathbf{Id}$ is the wide subcategory of $\mathbf{Fun}$ containing only identity functions. The third column tells us what a weakly $\mathfrak{S}$-combinatory complete applicative system in an $\mathfrak{S}$-multicategory is. For example, the second row states that an applicative system in a $\mathbf{Bij}$-multicategory is weakly $\mathbf{Bij}$-combinatory complete if and only if it is a $\mathsf{BCI}$-algebra. 

\begin{figure}
\begin{center}
\renewcommand{\arraystretch}{1.2}
\setlength{\tabcolsep}{1em}
\begin{tabular}{llll}
\toprule
Club $\mathfrak{S}$ & Short Name & Characterises \\
\midrule
  Identities &  $\mathbf{Id}$ & $\mathsf{BI}$-algebras \\ 
  Bijections & $\mathbf{Bij}$ & $\mathsf{BCI}$-algebras \\
  Monotone Injections & $\mathbf{Minj}$ & $\mathsf{BKI}$-algebras \\
  Injections & $\mathbf{Inj}$ & $\mathsf{BCKI}$-algebras \\
  Surjections &$\mathbf{Srj}$ & $\mathsf{BCWI}$-algebras \\
  Functions & $\mathbf{Fun}$ & $\mathsf{BCKWI}$-algebras \\
\bottomrule
\end{tabular}
\end{center}
\caption{Table of weak combinatory completeness results.}\label{fig:results-summary}
\end{figure}

The second contribution of this paper is a characterisation of $\mathfrak{S}$-combinatory complete applicative systems in terms of multicategory structure. Specifically, we show that an applicative system in an $\mathfrak{S}$-multicategory $\mathcal{M}$ is $\mathfrak{S}$-combinatory complete if and only if its computable maps form a sub-$\mathfrak{S}$-multicategory of $\mathcal{M}$ (Theorem~\ref{thm:combinatory-complete-multicategory}). This is rather satisfying: while the classical notion of combinatory completeness involved in, e.g., Theorem~\ref{thm:classical-combinatory-complete} can seem ad-hoc, Theorem~\ref{thm:combinatory-complete-multicategory} shows that it is in fact entirely natural when viewed from the right perspective.

As an intermediate step, we give a simple characterisation of the difference between applicative systems that are $\mathfrak{S}$-combinatory complete and applicative systems that are (merely) weakly $\mathfrak{S}$-combinatory complete. Specifically, a weakly $\mathfrak{S}$-combinatory complete applicative system $(A,\bullet)$ is $\mathfrak{S}$-combinatory complete if and only if for all $a \in A$ there exists some $a^\bullet \in A$ such that for all $x \in A$ we have $a^\bullet x = xa$ (Theorem~\ref{thm:flipped-complete}). We call such an applicative system \emph{flipped}, noting that a flipped $\mathsf{BI}$-algebra is precisely a $\mathsf{BI^\bullet}$-algebra in the sense of Tomita~\cite{Tomita2021}, who was the first to study them.

Given that combinatory completeness corresponds to multicategory structure on the computable maps, one might reasonably wonder when the multicategory of computable maps has this or that property. We resolve one such question, showing that for a given combinatory complete applicative system the multicategory of computable maps is \emph{closed} in the sense of Manzyuk~\cite{Manzyuk2012} precisely when the applicative system is extensional in an appropriate sense (Theorem~\ref{thm:closed-extensional}). 

One might also ask precisely \emph{which} multicategories arise as the computable maps of some combinatory complete applicative system. We show that it is precisely the \emph{weakly-closed operads} that arise in this way (an operad is a multicategory with exactly one object). Explicitly, the computable maps of any $\mathfrak{S}$-combinatory complete applicative system define a weakly-closed $\mathfrak{S}$-operad, and conversely the single object of a given weakly-closed $\mathfrak{S}$-operad is the carrier of a canonical $\mathfrak{S}$-combinatory complete applicative system therein. This extends to a coreflective adjunction between suitable categories of weakly-closed $\mathfrak{S}$-operads and of $\mathfrak{S}$-combinatory complete applicative systems in some ambient $\mathfrak{S}$-multicategory, witnessing the former as a full subcategory of the latter (Theorem~\ref{thm:adjunction}). In particular, this suggests that weakly-closed operads should be viewed as combinatory complete applicative systems ``standing alone'', or alternatively as combinatory complete applicative systems in the canonical context given by the associated multicategory of computable maps.

If we restrict our attention to the $\mathbf{Fun}$-multicategory $\mathsf{Set}$ with sets as objects and with functions $f : A_1 \times \cdots \times A_n \to B$ as morphisms $f \in \mathsf{Set}(A_1,\ldots,A_n;B)$ then we recover the classical notion of applicative system and of the $\mathsf{B},\mathsf{C},\mathsf{K},\mathsf{W}$, and $\mathsf{I}$ combinators. All of our combinatory completeness results specialise to the classical setting. For example, we obtain Theorem~\ref{thm:classical-combinatory-complete} as an instance of the fact that an applicative system in a $\mathbf{Fun}$-multicategory is $\mathbf{Fun}$-combinatory complete if and only if it is a $\mathsf{BCKWI}$-algebra. In this way, our results apply to the classical notion of applicative system as a set equipped with a binary operation.

This is a revised and expanded version of an earlier conference paper~\cite{Kuzmin2026}. More precisely: while Section~\ref{sec:structured-multicategories} and Section~\ref{sec:results} consist mainly of material from our earlier work~\cite{Kuzmin2026}; Section~\ref{sec:multicategory-structure}, Section~\ref{sec:closed-extensional}, and Section~\ref{sec:weakly-closed-operads} are novel.

\subsection{Related Work}
While the idea of a multicategory has arisen independently a number of times, the work of Lambek~\cite{Lambek1969} is most closely aligned with our purposes here. A more in-depth discussion of the multiple origins of the notion of multicategory may be found in Leinster's treatment~\cite[Chapter 2]{Leinster2004}, which is also an excellent reference in general. The sequent calculus presentation of multicategories originates with Lambek~\cite{Lambek1969,Lambek1989}, but see also the work of Szabo~\cite{Szabo1974,Szabo1971} on the subject. For closed multicategories see Manzyuk~\cite{Manzyuk2012}.

The notions of faithful cartesian club and structured multicategory that we consider in this paper first appear, essentially, in the work of Tronin~\cite{Tronin2002} under the names of ``verbal category'' and ``W-operad'', respectively. Here we follow  Shulman~\cite{Shulman2016} in both our terminology and technical development. Faithful cartesian clubs are an instance of the more general notion of club introduced by Kelly~\cite{Kelly1972,Kelly1972a}. See also the work of Crutwell and Shulman on generalised multicategories~\cite{Crutwell2010}. We found the treatment of structured multicategories in the work of Kr\"{a}mer and Mahaman~\cite{Krahmer2021} to be helpful.

The work presented here grew out of an interest in the work of Cockett and Hofstra on Turing categories~\cite{Cockett2008}, which contains a combinatory completeness result for \emph{partial} combinatory algebras that is similar in spirit to the results presented herein. Moreover, weakly-closed structure plays an important role in the theory of Turing categories, and the situation there has directly inspired the results concerning weakly-closed operads herein. Perhaps more directly relevant is the work of Longo and Moggi~\cite{Longo1990}, which contains a version of Theorem~\ref{thm:classical-combinatory-complete} internal to categories with finite products. We have also been inspired be the work of Hyland~\cite{Hyland2017} and Hasegawa~\cite{Hasegawa2023} on semi-closed multicategories and models of the lambda-calculus.

Finally, we credit the notion of flipped applicative system to Tomita~\cite{Tomita2021,Tomita2022,Tomita2023}, who has studied flipped $\mathsf{BI}$-algebras under the name of $\mathsf{BI}^\bullet$-algebras.

\section{Structured Multicategories}\label{sec:structured-multicategories}
In this section we recapitulate some necessary background material concerning structured multicategories, the attendant notion of faithful cartesian club, and the connection of structured multicategories to sequent calculus. In doing so, we closely follow Shulman~\cite{Shulman2016}.

\subsection{Faithful Cartesian Clubs}\label{subsec:faithful-cartesian-clubs}
The notion of faithful cartesian club revolves around the category of finite ordinals and functions, which we introduce now. For $n \in \mathbb{N}$, write $\underline{n} = \{1,\ldots,n\}$. Let $\mathbf{Fun}$ be the category with natural numbers as objects, and with morphisms $\mathbf{a} : n \to m$ given by functions $\mathbf{a} : \underline{n} \to \underline{m}$. Composition and identities are given by function composition and identity functions, respectively.

Morphisms of $\mathbf{Fun}$ can be understood intuitively as ``dot and line'' diagrams. For example, the morphism indicated below on the right corresponds to the diagram below on the left:
\begin{mathpar}
  \includegraphics[height=1.5cm,align=c]{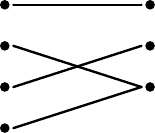} % 85px high

  \mathbf{a}(x) =
  \begin{cases}
    1 & \text{if } x = 1\\
    2 & \text{if } x = 3\\
    3 & \text{otherwise}
  \end{cases}

  \mathbf{a} : 4 \to 3
\end{mathpar}
The set $\underline{n}$ is depicted as a sequence of $n$ dots, with the uppermost dot corresponding to $1 \in \underline{n}$, the dot below it corresponding to $2 \in \underline{n}$, and so on. The action of the function is indicated by lines connecting each element of the domain to the element of the codomain that it is mapped to. We note that in general such diagrams indicate a \emph{relation} between finite ordinals, but may nonetheless be used to discuss functions.

We will be interested in the monoidal category structure $(\mathbf{Fun},+,0)$ on $\mathbf{Fun}$ in which $+$ is defined on objects as the usual addition of natural numbers, and is defined on morphisms $\mathbf{a} : n \to m$ and $\mathbf{b} : h \to k$ as in:
\begin{mathpar}
  (\mathbf{a}+\mathbf{b})(x) =
  \begin{cases}
    \mathbf{a}(x) & \text{if } x \leq n\\
    \mathbf{b}(x - n) + m & \text{if } x > n
  \end{cases}
\end{mathpar}
In terms of dot and line diagrams, $\mathbf{a}+\mathbf{b}$ is obtained by vertically ``concatenating'' the diagrams for $\mathbf{a}$ and $\mathbf{b}$, as in:
\begin{mathpar}
  \includegraphics[height=1cm,align=c]{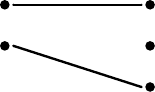} %height 59px
  \hspace{0.3cm}+\hspace{0.3cm}
  \includegraphics[height=1cm,align=c]{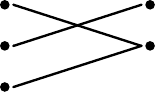} %height 59px
  \hspace{1cm}
  =
  \hspace{1cm}
  \includegraphics[height=1.95cm,align=c]{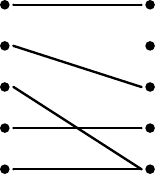} %height 111px
\end{mathpar}

It so happens that $(\mathbf{Fun},+,0)$ is cocartesian monoidal, so that $\sum_{i \in \underline{n}} k_i$ is a coproduct of $k_1,\ldots,k_n \in \mathbb{N}$. We will respectively write
\[
    \textstyle
    \ip_j^{(k_i)_{i \in \underline{n}}} : k_j \to ( \sum_{i \in \underline{n}} k_i ) 
        \quad\quad \text{and} \quad\quad
    \langle \mathbf{a}_1 \mid \ldots \mid \mathbf{a}_n \rangle : ( \sum_{i\in \underline{n}} k_i )  \to m
\]
for the coproduct injections and the copairing of the family $(\mathbf{a}_i : k_i \to m)_{i \in \underline{n}}$.

In particular, for each object $n$ of $\mathbf{Fun}$ we have $n = ( \sum_{i \in \underline{n}} 1 )$, and so any morphism $\mathbf{a} : m \to n$ can be written as a copairing
\[
    \mathbf{a} = \left\langle \ip_{\mathbf{a}(1)}^{(1)_{i\in \underline{n}}} \bigg| \cdots \bigg| \ip_{\mathbf{a}(m)}^{(1)_{i\in \underline{n}}} \right\rangle
\]
of coproduct injections. The fact that every morphism of $\mathbf{Fun}$ can be written in this way is helpful in defining an operation called the \emph{wreath product}. Specifically, for each $\mathbf{a} : m \to n$ in $\mathbf{Fun}$ and $k_1,\ldots,k_n \in \mathbb{N}$ we define their wreath product to be a morphism $\mathbf{a} \wr (k_1,\ldots,k_n) : ( \sum_{j \in \underline{m}} k_{\mathbf{a}(j)}) \to (\sum_{i \in \underline{n}} k_i)$ of $\mathbf{Fun}$ as in:
\[
  \mathbf{a} \wr (k_1,\ldots,k_n)
  = \left\langle \ip_{\mathbf{a}(1)}^{(1)_{i\in \underline{n}}} \bigg| \cdots \bigg| \ip_{\mathbf{a}(m)}^{(1)_{i\in \underline{n}}} \right\rangle \wr (k_1,\ldots,k_n)
  \coloneqq \left\langle \ip_{\mathbf{a}(1)}^{(k_i)_{i \in \underline{n}}} \bigg| \cdots \bigg| \ip_{\mathbf{a}(m)}^{(k_i)_{i \in \underline{n}}} \right\rangle
\]

The above definition is somewhat opaque. Fortunately, the effect of the wreath product is easily understood when we consider it in terms of dot and line diagrams. There the diagram representing $\mathbf{a} \wr (k_1,\ldots,k_n)$ is obtained from the diagram representing $\mathbf{a}$ by ``thickening'' it in amounts given by the $k_i$. First one thickens the codomain by replacing each dot $i \in \underline{n}$ with $k_i$ separate dots, so that the new codomain is $\sum_{i\in \underline{n}} k_i$, and thickens the domain by replacing each dot $j \in \underline{m}$ with $k_{\mathbf{a}(j)}$ separate dots, so that the new domain is $\sum_{j \in \underline{m}} k_{\mathbf{a}(j)}$. Finally one thickens the line leaving each dot $j \in \underline{m}$ into $k_{\mathbf{a}(j)}$ parallel lines, connecting the $k_{\mathbf{a}(j)}$ points that have replaced $j \in \underline{m}$ to the $k_{\mathbf{a}(j)}$ points that have replaced $\mathbf{a}(j) \in \underline{n}$. For example, if $\mathbf{a} : 4 \to 4$ is represented by the diagram below left, then $\mathbf{a} \wr (3,2,3,2) : 9 \to 10$ is represented by the diagram below right. 
\begin{mathpar} 
  \includegraphics[height=2cm,align=c]{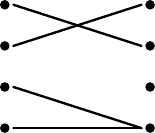} %height 84px

  \includegraphics[height=3cm,align=c]{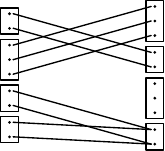} %height 96px 
\end{mathpar}

We are now ready to give the central definition of this section:
\begin{definition}[\cite{Shulman2016}]
A \emph{faithful cartesian club} is a subcategory $\mathfrak{S}$ of $\mathbf{Fun}$ such that:
\begin{itemize}
\item $\mathfrak{S}$ contains all the objects of $\mathbf{Fun}$ (i.e., it is a \emph{wide} subcategory).
\item If morphisms $\mathbf{a}$ and $\mathbf{b}$ are in $\mathfrak{S}$ then so is $\mathbf{a}+\mathbf{b}$ (i.e., it is closed under $+$).
\item If $\mathbf{a}: m \rightarrow n$ is in $\mathfrak{S}$ then so is $\mathbf{a} \wr (k_1, \ldots, k_n)$ for all $k_1, \ldots, k_n$ (i.e., it is closed under $\wr$).
\end{itemize}
\end{definition}

\begin{figure}
\begin{center}
\renewcommand{\arraystretch}{1.2}
\setlength{\tabcolsep}{1em}
\begin{tabular}{llll}
\toprule
Club $\mathfrak{S}$ & Consists of & Generated by \\
\midrule
  \textbf{Id} & identities & -- \\
  \textbf{Bij} & bijections & $\tau$ \\
  \textbf{Minj} & monotone injections & $\delta$ \\
  \textbf{Inj} & injections &  $\tau, \delta$ \\
  \textbf{Srj} & surjections &  $\tau, \sigma$ \\
  \textbf{Fun} & functions &  $\tau, \sigma, \delta$ \\
\bottomrule
\end{tabular}
\end{center}
\caption{Some faithful Cartesian clubs and their relationship to the face, degeneracy, and transposition maps.}\label{fig:clubs}
\end{figure}

The faithful Cartesian clubs that will be relevant to us here are listed in Figure~\ref{fig:clubs}. These clubs can be understood in terms of certain special classes of morphism in $\mathbf{Fun}$. Specifically, we consider:
\begin{itemize}
\item \emph{transpositions} $\tau_i^n : n \to n$ for all $n > 1$ and $1 \leq i < n$ defined as in:
  \begin{mathpar}
    \tau_i^n(x) =
    \begin{cases}
      x+1 &\text{if } x = i \\
      x-1 &\text{if } x = i+1 \\
      x &\text{otherwise}
    \end{cases}

    \includegraphics[height=1.5cm,align=c]{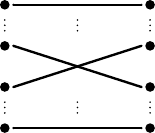} %height 85px
  \end{mathpar}
  so $\tau_i^n$ is the permutation that swaps the $i$th and $(i+1)$st elements.
\item \emph{degeneracy maps} $\sigma^n_i : n+1 \to n$ for all $n$ and $1 \leq i \leq n$ defined as in: 
  \begin{mathpar}
    \sigma^n_i(x) = 
    \begin{cases}
      x &\text{if } x \leq i \\
      x-1 &\text{if } x > i
    \end{cases}

    \includegraphics[height=1.5cm,align=c]{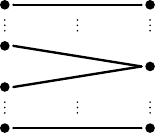} %height 85px
  \end{mathpar}
  so $\sigma^n_i$ is the monotone surjection that merges the $i$th and $(i+1)$st elements.
\item \emph{face maps} $\delta^n_i : n - 1 \to n$ for all $n \geq 1$ and $1 \leq i \leq n$, defined as in:
  \begin{mathpar}
    \delta^n_i(x)=
    \begin{cases}
      x &\text{if } x < i\\
      x + 1&\text{if }x \geq i
    \end{cases}

    \includegraphics[height=1.95cm,align=c]{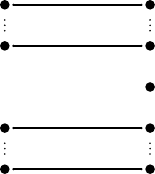} %height 112px
  \end{mathpar}
  so $\delta_i^n$ is the monotone injection that skips the $i$th element of the codomain.
\end{itemize}

Each of these clubs is generated by certain sorts of special morphism together with identity morphisms through composition. $\mathbf{Fun}$ itself is generated by all of the face, degeneracy, and transposition maps, $\mathbf{Bij}$ is generated by the transposition maps, and so on (see e.g.,~\cite{Grandis2001,Burroni1993}). The other characterisations are given in the third column of Figure~\ref{fig:clubs}.

Note that not all possible combinations of generators occur in Figure~\ref{fig:clubs}. In particular, the monotone surjections and monotone functions are not closed under the wreath product, and so do not form a faithful cartesian club. In particular, $\sigma_1^1 \wr (2)$ is not monotone.

\subsection{Multicategories and Structured Multicategories}\label{subsec:structured-multicategories}

We proceed to define multicategories, and to define $\mathfrak{S}$-multicategories for a faithful cartesian club $\mathfrak{S}$.  Multicategories are similar to categories, the primary difference being that while the domain of a morphism in a category consists of a single object, the domain of a morphism in a multicategory consists of a finite sequence of objects. Explicitly:

\begin{definition}[\cite{Leinster2004}]\label{def:multicategory}
A \emph{multicategory} $\mathcal{M}$ consists of the following data:
\begin{itemize}
  \item A set $\mathcal{M}_0$ whose elements are called the \emph{objects} of $\mathcal{M}$.
  \item For each $n \in \mathbb{N}$ and $A_1, \ldots, A_n, B \in \mathcal{M}_0$, 
        a set $\mathcal{M}(A_1, \ldots, A_n; B)$ of \emph{morphisms}. Any $f \in \mathcal{M}(A_1,\ldots,A_n;B)$ is said to \emph{have arity $n$}.
  \item For each $n \in \mathbb{N}$, $A_1,\ldots,A_n,B \in \mathcal{M}_0$,  and $\Gamma_1,\ldots,\Gamma_n \in \mathcal{M}_0^*$, a \emph{composition} operation:
        \[
          \mathcal{M}(A_1,\ldots,A_n;B) \times \mathcal{M}(\Gamma_1;A_1) \times \cdots \times \mathcal{M}(\Gamma_n;A_n)
          \stackrel{\circ}{\longrightarrow}
          \mathcal{M}(\Gamma_1,\ldots,\Gamma_n;B)
        \]
        which we will usually write infix as in $f \circ (g_1,\ldots,g_n) = \circ(f,g_1,\ldots,g_n)$.
  \item For each $A \in \mathcal{M}_0$, an \emph{identity} morphism $1_A \in \mathcal{M}(A; A)$.
      \end{itemize}
      This data must be such that:
      \begin{itemize}
  \item Composition is associative. That is, we have:
    \begin{align*}
      & f \circ  (g_1 \circ (h_1^1, \ldots, h_1^{k_1}), \ldots, 
      g_n \circ (h_n^1, \ldots, h_n^{k_n})) \\
      &=  (f \circ  (g_1, \ldots, g_n)) \circ (h_1^1, \ldots, h_1^{k_1}, \ldots, h_n^1, \ldots, h_n^{k_n}) 
    \end{align*}
    whenever $f,g_i,h_i^j$ are morphisms for which the composites make sense.
  \item Identity morphisms are unital. That is, we have:
    \[
      f \circ (1_{A_1}, \ldots, 1_{A_n}) = f = 1_B \circ f
    \]
    for every $f \in \mathcal{M}(A_1, \ldots, A_n;B)$.
\end{itemize}
\end{definition}
If $A$ is an object of a multicategory $\mathcal{M}$ we write $A^n$ to indicate the sequence $A,\ldots,A$ consisting of $n$ copies of $A$. Similarly, if $f \in \mathcal{M}(\Gamma;A)$ we write $f^n$ to indicate the sequence consisting of $n$ copies of $f$. For example, if $g \in \mathcal{M}(A^n;B)$ then we may write $g \circ (f^n)$ to indicate the composite $g \circ (f,\ldots,f)$. It is important to note that morphisms of a multicategory $\mathcal{M}$ may have arity $0$, in which case their domain is the empty sequence as in $\mathcal{M}(;B)$.

Now, structured multicategories are defined as in:
\begin{definition}[\cite{Shulman2016}]
Let $\mathfrak{S}$ be a faithful cartesian club. An \emph{$\mathfrak{S}$-multicategory} is a multicategory $\mathcal{M}$ together with an operation:
\[
  \mathcal{M}(A_{\mathbf{a}(1)}, \ldots, A_{\mathbf{a}(m)}; B) \xrightarrow{[-]\mathbf{a}} \mathcal{M}(A_1, \ldots, A_n;B) 
\]
for each $\mathbf{a} : m \to n$ in $\mathfrak{S}$, such that:
\begin{itemize}
    \item $[[f]\mathbf{a}]\mathbf{b} = [f](\mathbf{b} \circ \mathbf{a})$
    \item $[f]1_{\underline{n}} = f$
    \item $g \circ ([f_1]\mathbf{a}_1, \ldots, [f_n]\mathbf{a}_n) = [g \circ (f_1, \ldots, f_n)](\mathbf{a}_1 + \cdots + \mathbf{a}_n)$
    \item $[g]\mathbf{a} \circ (f_1, \ldots, f_n) = [g \circ (f_{\mathbf{a}(1)}, \ldots, f_{\mathbf{a}(m)})](\mathbf{a} \wr (k_1, \ldots, k_n))$ where each $k_i$ is the arity of $f_i$.
\end{itemize}
\end{definition}
  
A good way to understand morphisms in multicategories is through their string diagrams. For example a morphism $g$ of arity $4$ in some multicategory is pictured below left. Composites may be depicted by merging input and output wires. For example if in addition to $g$ we have morphisms $f$ and $h$ of arity $2$ and $3$ respectively, then $g \circ (f,h,f,f)$ is pictured below right.
\begin{mathpar}
  \includegraphics[height=1.7cm,align=c]{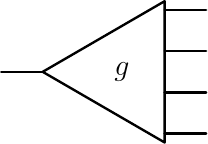} %height 91px

   \includegraphics[height=1.7cm,align=c]{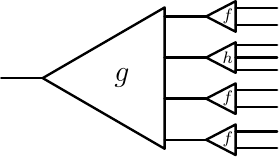} %height 164px
\end{mathpar}
In $\mathfrak{S}$-multicategories, it is convenient to depict the action of a given $\mathbf{a} \in \mathfrak{S}$ by juxtaposing the dot and line diagram for $\mathbf{a}$ with the string diagram depicting $f$. For example, if our morphism $g$ inhabits a $\mathbf{Fun}$-multicategory and $\mathbf{a} : 4 \to 4 \in \mathbf{Fun}$ is the function depicted below left, then the morphism $[g]\mathbf{a}$ is pictured below right.
\begin{mathpar}
  \includegraphics[height=1.5cm,align=c]{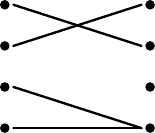} %height 84px

  \includegraphics[height=1.7cm,align=c]{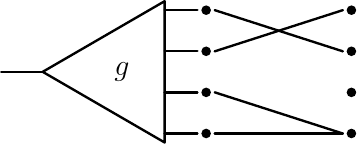} %height 91px
\end{mathpar}
It can be helpful to consider the axioms of an $\mathfrak{S}$-multicategory from this perspective. In particular, the axiom concerning whiskering tells us that:
\begin{align*}
  &[g]\mathbf{a} \circ (h,f,h,f)
  = [g \circ (f,h,f,f)](\mathbf{a}\wr (3,2,3,2))
\end{align*}
which is pictured as in:
\begin{mathpar}
  \includegraphics[height=1.8cm,align=c]{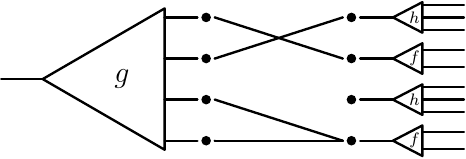} %height 100px
  \hspace{0.3cm}
  =
  \hspace{0.3cm}
  \includegraphics[height=1.8cm,align=c]{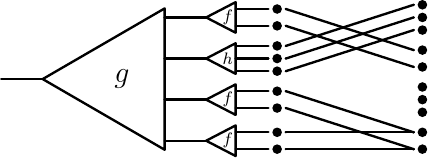} %height 100px
\end{mathpar}

While some of the notions of structured multicategory we consider are a little exotic, we note that $\mathbf{Id}$-multicategories are just multicategories, that a $\mathbf{Bij}$-multicategory is what is usually called a \emph{symmetric multicategory}, and that a $\mathbf{Fun}$-multicategory is what is usually called a \emph{cartesian multicategory}.

Before moving on, we record a few auxiliary definitions that will be required later on in our development. First, we will require a notion of subcategory for structured multicategories:
\begin{definition}
  Let $\mathcal{M}$ be a multicategory. A \emph{sub-multicategory} $\mathcal{N}$ of $\mathcal{M}$ consists of:
  \begin{itemize}
  \item A subcollection $\mathcal{N}_0 \subseteq \mathcal{M}_0$ of \emph{objects}.
  \item For each $A_1,\ldots,A_n,B \in \mathcal{N}_0$, a subcollection
    \[
      \mathcal{N}(A_1,\ldots,A_n;B) \subseteq \mathcal{M}(A_1,\ldots,A_n;B)
    \]
    of \emph{morphisms}.
  \end{itemize}
  such that:
  \begin{itemize}
  \item $\mathcal{N}$ has identity morphisms. That is, for each $A \in \mathcal{N}_0$ we have $1_A \in \mathcal{N}(A;A)$.
  \item $\mathcal{N}$ is closed under composition. That is, if $g \in \mathcal{N}(A_1,\ldots,A_n;B)$ and $(f_i \in  \mathcal{N}(\Gamma_i;A_i))_{i=1}^n$ then $g \circ (f_1,\ldots,f_n) \in \mathcal{N}(\Gamma_1,\ldots,\Gamma_n;B)$.
  \end{itemize}
  where the composition and identity morphisms above are those of $\mathcal{M}$. In case $\mathcal{M}$ is an $\mathfrak{S}$-multicategory for some faithful cartesian club $\mathfrak{S}$, we say that $\mathcal{N}$ is a \emph{sub-$\mathfrak{S}$-multicategory} of $\mathcal{M}$ in case:
  \begin{itemize}
  \item  $\mathcal{N}$ closed under the action of $\mathfrak{S}$. That is, if $\mathbf{a} : m \to n \in \mathfrak{S}$ and  $f \in \mathcal{N}(A_{\mathbf{a}(1)},\ldots,A_{\mathbf{a}(m)};B)$, then $[f]\mathbf{a} \in \mathcal{N}(A_1,\ldots,A_n;B)$.
  \end{itemize}
\end{definition}

Next, we define a functor of structured multicategories:
\begin{definition}
  Let $\mathcal{M}$ and $\mathcal{N}$ be multicategories. A \emph{multifunctor} $F : \mathcal{M} \to \mathcal{N}$ consists of:
  \begin{itemize}
  \item A mapping $F_0 : \mathcal{M}_0 \to \mathcal{N}_0$.
  \item A mapping 
    \[F : \mathcal{M}(A_1,\ldots,A_n;B) \to \mathcal{N}(FA_1,\ldots,FA_n;FB)\]
    for each $A_1,\ldots,A_n,B \in \mathcal{M}_0$.
  \end{itemize}
  which must preserve composition and identities as in:
  \begin{itemize}
  \item $F(f) \circ (F(g_1),\ldots,F(g_n)) = F(f \circ (g_1,\ldots,g_n))$
  \item $F(1_A) = 1_{FA}$
  \end{itemize}
  In case $\mathcal{M}$ and $\mathcal{N}$ are $\mathfrak{S}$-multicategories for some faithful cartesian club $\mathfrak{S}$, we call $F$ an \emph{$\mathfrak{S}$-multifunctor} in case it preserves the action of $\mathfrak{S}$ as in:
  \begin{itemize}
  \item $F([f]\mathbf{a}) = [F(f)]\mathbf{a}$
  \end{itemize}
\end{definition}

\subsection{Sequent Calculus}\label{subsec:sequent-calculus}

Structured multicategories enjoy a deep connection to intuitionistic sequent calculus. Let $\Sigma$ be a (multi-sorted) signature in which operation symbols are typed as in $f : A_1,\ldots,A_n \vdash B$ where $A_1,\ldots,A_n,B$ are generating sorts. Then a \emph{term} over $\Sigma$ is a sequent that is derivable via the following inference rules:
\begin{mathpar}
  \inferrule*[right=var]{\text{ }}{x:A \vdash x:A}

  \inferrule*[right=op]{(\Gamma_i \vdash t_i : A_i)_{i \in \{1,\ldots,n\}} \\ (f : A_1,\ldots,A_n \vdash B) \in \Sigma}{\Gamma_1,\ldots,\Gamma_n \vdash f(t_1,\ldots,t_n) : B}
\end{mathpar}

Crucially, for a sequent $\Gamma \vdash t : B$ to be considered well-formed, the context $\Gamma$ must not contain any repeated variables. This means that, for example, when we write $\Gamma_1,\ldots,\Gamma_n$, it is implied that the variables in the $\Gamma_i$ are disjoint.

Terms over a signature form a multicategory, with identity morphisms given by the \textsc{var} rule and with composition given by substitution. More precisely, the composition operation is given by the following admissible inference rule:
\begin{mathpar}
  \inferrule*[right=comp]{(\Gamma_i \vdash t_i : A_i)_{i \in \{1,\ldots,n\}} \\ x_1 : A_1,\ldots,x_n : A_n \vdash t : B}{\Gamma_1,\ldots,\Gamma_n \vdash t[t_1,\ldots,t_n/x_1,\ldots,x_n] : B}
\end{mathpar}
That this satisfies the equations of a multicategory follows from certain elementary properties of substitution. For example, for any $\Gamma \vdash t : B$ the right-unitality law ($1_B \circ f = f$) holds as in:
\[
  (\Gamma \vdash x[t/x] : B) = (\Gamma \vdash t : B)
\]
We omit the redundant information when writing such equations, so that for example $\Gamma \vdash x[t/x] = t : B$ is an equivalent way of expressing the above equation.

The multicategory of terms over a signature is in fact the \emph{free} multicategory over that signature, in the sense that this construction gives the left adjoint of an adjunction between a category of signatures and the category of multicategories. The right adjoint maps a multicategory $\mathcal{M}$ to the $\mathcal{M}_0$-sorted signature with an operation symbol $f : A_1,\ldots,A_n \vdash B$ for each $f \in \mathcal{M}(A_1,\ldots,A_n;B)$. The counit of the adjunction gives a morphism from the multicategory of terms over this signature into $\mathcal{M}$, and quotienting the terms by the equations that hold in the image of this morphism yields a sequent calculus presentation of $\mathcal{M}$. It follows that we may reason about morphisms in any multicategory by means of sequent calculus, interpreting composition as substitution and identities as variables.

To extend this to structured multicategories requires an additional structural inference rule. For a given faithful cartesian club $\mathfrak{S}$ we ask that:
\begin{mathpar}
  \inferrule*[right=act]{\mathbf{a} : m \to n \in \mathfrak{S} \\ x_1 : A_{\mathbf{a}(1)},\ldots,x_m : A_{\mathbf{a}(m)} \vdash t : B}{x_1 : A_1,\ldots,x_n : A_n \vdash [t]\mathbf{a} : B}
\end{mathpar}
where the term $[t]\mathbf{a}$ is defined as in:
\begin{mathpar}
  \text{}
  [t]\mathbf{a} =
  \begin{cases}
    x_{\mathbf{a}(i)} & \text{if } t = x_i \text{ is a variable}\\
    f([t_1]\mathbf{a},\ldots,[t_n]\mathbf{a})  & \text{if } t = f(t_1,\ldots,t_n)
  \end{cases}
\end{mathpar}
Note that $[t]\mathbf{a}$ is a meta-level operation, just like substitution. Now the derivable sequents form the free $\mathfrak{S}$-multicategory, and as before one obtains a sequent calculus presentation for arbitrary $\mathfrak{S}$-multicategories.

The effect of the $\textsc{act}$ rule is to allow the variables of the context to be used more flexibly in terms, with the degree of flexibility depending on $\mathfrak{S}$. In the sequent calculus for multicategories, derivable sequents $\Gamma \vdash t : B$ have the property that the variables of the context $\Gamma$ occur exactly once in $t$, in exactly the same order they appear in $\Gamma$. For example, we can derive $x_1 : A_1, x_2 : A_2 \vdash f(x_1,x_2) : B$ but not $x_1 : A_1, x_2 : A_2 \vdash f(x_2,x_1) : B $ or $x_1 : A_1 \vdash f(x_1,x_1) : B$. Similarly, while we can derive $x_1 : A_1 \vdash g(x_1) : B$, we cannot derive $x_1 : A_1, x_2 : A_2 \vdash g(x_1) : B$. In a $\mathbf{Fun}$-multicategory all of these are possible as in:
\begin{mathpar}
  \inferrule{\tau_1^2 : 2 \to 2 \in \mathbf{Fun} \\ x_1:A_1,x_2:A_2 \vdash f(x_1,x_2) : B}{x_1 : A_1,x_2 : A_2 \vdash [f(x_1,x_2)]\tau_1^2 = f(x_2,x_1) : B} 

  \inferrule{\sigma_1^1 : 2 \to 1 \in \mathbf{Fun} \\ x_1:A_1 \vdash f(x_1,x_2) : B}{x_1:A_1 \vdash [f(x_1,x_2)]\sigma_1^1 = f(x_1,x_1) : B}

  \inferrule{\delta_1^2 : 1 \to 2 \in \mathbf{Fun} \\ x_1:A_1 \vdash g(x_1) : B}{x_1:A_1, x_2:A_2 \vdash [g(x_1)]\delta_1^2 = g(x_1) : B}
\end{mathpar}

This is a good way to think about the different notions of polynomial obtained from the different faithful cartesian clubs. Different sorts of combinatory completeness correspond to different sorts of restriction on the use of variables in terms. Formally, terms represent morphisms in some $\mathfrak{S}$-multicategory, and the choice of $\mathfrak{S}$ determines how variables may be used.

\section{Weak Combinatory Completeness and Combinators}\label{sec:results}
In this section we reformulate the basic definitions surrounding applicative systems and combinatory completeness in the setting of a (possibly structured) multicategory, give a general notion of weak $\mathfrak{S}$-combinatory completeness with respect to a faithful cartesian club $\mathfrak{S}$, and use it to characterise the existence of certain collections of combinators in a given applicative system, as summarised in Figure~\ref{fig:results-summary}. We begin with the notion of applicative system:
\begin{definition}
  Let $\mathcal{M}$ be a multicategory. An \emph{applicative system} $(A,\application)$ in $\mathcal{M}$ consists of an object $A \in \mathcal{M}_0$ together with a morphism $\application \in \mathcal{M}(A,A;A)$, called \emph{application}.
\end{definition}
When working with applicative systems in multicategories, we will tend to favour the sequent calculus syntax discussed in Section~\ref{sec:structured-multicategories}, and when doing so will adopt the usual syntactic conventions for working with application discussed in Section~\ref{sec:intro}, so that for example $xyz = (xy)z = (x \application y) \application z$. Moreover, when we are working with an applicative system $(A,\application)$ in a multicategory $\mathcal{M}$, all of the relevant morphisms are elements of $\mathcal{M}(A^n;A)$ for some $n \in \mathbb{N}$. This allows us to omit the types from our sequents, since everything has type $A$. For example we may write $x_1,x_2 \vdash f(x_1,x_2)$ instead of $x_1 : A , x_2 : A \vdash f(x_1,x_2) : A$, which helps to make things less cluttered. We will moreover allow ourselves to use variable names beyond $x_i$. While this can be made fully formal, we refrain from doing so here. The interested reader may consult the appendix of~\cite{Shulman2016}, which follows the ``nominal'' approach of Gabbay and Pitts~\cite{Gabbay1999}.

If $(A,\application)$ is an applicative system in $\mathcal{M}$ then we define an \emph{iterated application} operation $\application^n \in \mathcal{M}(A,A^n;A)$ for each $n \in \mathbb{N}$ as in $\application^0 = 1_A$ and $\application^{n+1} = \application \circ (\application^n,1_A)$. Note that $\bullet^1 = \bullet \circ (1_A,1_A) = \bullet$, and that in the sequent calculus notation $\application^n$ becomes $x,x_1,\ldots,x_n \vdash xx_1\cdots x_n$. This facilitates the following definition:
\begin{definition}
  Let $\mathcal{M}$ be a multicategory and let $(A,\application)$ be an applicative system in $\mathcal{M}$. We say that $f \in \mathcal{M}(A^n;A)$ is \emph{$(A,\application)$-computable} in case there exists some $a \in \mathcal{M}(;A)$ in $\mathcal{M}$ such that $\application^n \circ (a,1_A,\ldots,1_A) = f$ or, equivalently, such that $x_1,\ldots,x_n \vdash ax_1\cdots x_n = f(x_1,\ldots,x_n)$.
\end{definition}
The morphisms $a \in \mathcal{M}(;A)$ play the role of elements of $A$, and are all $(A,\bullet)$-computable as in $\bullet^0 \circ (a) = 1_A \circ (a) = a$.

Next, every faithful cartesian club gives a notion of regular polynomial as follows:
\begin{definition}\label{def:regular-s-polynomial}
  Let $\mathfrak{S}$ be a faithful cartesian club, and let $(A,\application)$ be an applicative system in an $\mathfrak{S}$-multicategory $\mathcal{M}$. Let $\mathfrak{S}_\mathsf{r}(A,\bullet)$ be the smallest sub-$\mathfrak{S}$-multicategory of $\mathcal{M}$ containing $\application \in \mathcal{M}(A,A;A)$. We refer to morphisms of $\mathfrak{S}_\mathsf{r}(A,\bullet)$ as \emph{regular $\mathfrak{S}$-polynomials over $(A,\bullet)$}.
\end{definition}

For example, $x_1,x_2,x_3 \vdash x_1(x_2x_3)$ is a regular $\mathfrak{S}$-polynomial over $(A,\bullet)$. Adapting the classical notion of combinatory completeness, we obtain:
\begin{definition}\label{def:weak-s-combinatory-complete}
  Let $\mathfrak{S}$ be a faithful cartesian club, $\mathcal{M}$ be an $\mathfrak{S}$-multicategory, and $(A,\application)$ be an applicative system in $\mathcal{M}$. We say that $(A,\bullet)$ is \emph{weakly $\mathfrak{S}$-combinatory complete} in case every morphism of $\mathfrak{S}_\mathsf{r}(A,\bullet)$ is $(A,\bullet)$-computable.
\end{definition}

We proceed to establish our results, beginning with the simplest. We define:
\begin{definition}
  Let $\mathcal{M}$ be a multicategory (i.e., an $\mathbf{Id}$-multicategory), and let $(A,\application)$ be an applicative system in $\mathcal{M}$.
  \begin{itemize}
  \item A \emph{$\mathsf{B}$ combinator} for $(A,\application)$ is $\mathsf{B} \in \mathcal{M}(;A)$ such that $\application^3 \circ (\mathsf{B},1_A,1_A,1_A) = \application \circ (1_A,\application)$, or equivalently $x_1,x_2,x_3 \vdash \mathsf{B}x_1x_2x_3 = x_1(x_2x_3)$.
  \item An \emph{$\mathsf{I}$ combinator} for $(A,\application)$ is a morphism $\mathsf{I} \in \mathcal{M}(;A)$ such that $\application \circ (\mathsf{I},1_A) = 1_A$, or equivalently $x \vdash \mathsf{I}x = x$.
  \end{itemize}
  If $(A,\application)$ has both a $\mathsf{B}$ and $\mathsf{I}$ combinator, we say that it is a \emph{$\mathsf{BI}$-algebra}.
\end{definition}

\begin{lemma}\label{lem:iterated-b}
  Let $(A,\application)$ be a $\mathsf{BI}$-algebra in a multicategory $\mathcal{M}$. We define a morphism $\mathsf{B}^n \in \mathcal{M}(;A)$ for each $n \in \mathbb{N}$ as follows: $\mathsf{B}^0 = \mathsf{I}$, $\mathsf{B}^1 = \mathsf{B}$, and $\mathsf{B}^{n+1} = \mathsf{BB}(\mathsf{B}^n)$ for $n \ge 1$. Then for all $n \in \mathbb{N}$ we have  $\application^{n+1} \circ (\mathsf{B}^n,1_A,1_A,1_A^n) = \application \circ (1_A,\application^n)$, or equivalently, $b,a,x_1,\ldots,x_n \vdash \mathsf{B}^nbax_n\cdots x_1 = b(a x_n\cdots x_1)$. 
\end{lemma}
\begin{proof}
  By induction on $n \in \mathbb{N}$. The base cases are when $n = 0$, in which case we have $b,a \vdash \mathsf{B}^0ba = \mathsf{I}ba = ba$, and when $n = 1$, in which case we have $b,a,x_1 \vdash \mathsf{B}^1bax_1 = \mathsf{B}bax_1 = b(ax_1)$. For the inductive case, suppose that we have $b,a,x_n,\ldots,x_1 \vdash \mathsf{B}^nbax_n \cdots x_1 = b(ax_n\cdots x_1)$. Then we also have:
  \begin{align*}
    &b,a,x_{n+1},x_n,\ldots,x_1
    \vdash \mathsf{B}^{n+1}b a x_{n+1} x_n \cdots x_1
      = \mathsf{BB}(\mathsf{B}^n)b a x_{n+1} x_n \cdots x_1
      \\&= \mathsf{B}(\mathsf{B}^nb) a x_{n+1} x_n \cdots x_1
      = (\mathsf{B}^n b)(a x_{n+1}) x_n \cdots x_1
      = \mathsf{B}^n b(a x_{n+1}) x_n \cdots x_1
      \\&= b((a x_{n+1}) x_n \cdots x_1)
      = b(a x_{n+1} x_n \cdots x_1)
  \end{align*}
  The claim follows by induction.
\end{proof}

Now our combinatory completeness result for $\mathsf{BI}$-algebras is as follows:
\begin{theorem}\label{thm:id-combinatory-complete}
  Let $(A,\application)$ be an applicative system in a multicategory $\mathcal{M}$. Then $(A,\bullet)$ is weakly $\mathbf{Id}$-combinatory complete if and only if it is a $\mathsf{BI}$-algebra.
\end{theorem}
\begin{proof}
  Suppose $(A,\application)$ is weakly $\mathbf{Id}$-combinatory complete. Then $\application \circ (1_A,\application)$ and $1_A$ are $(A,\application)$-computable, since they are both morphisms of $\mathbf{Id}_\mathsf{r}(A,\application)$. That is, there exist $\mathsf{B} \in \mathcal{M}(;A)$ and $\mathsf{I} \in \mathcal{M}(;A)$ such that $\application^3 \circ (B,1_A,1_A,1_A) = \application \circ (1_A,\application)$ and $\application \circ (\mathsf{I},1_A) = 1_A$. It follows that $(A,\application)$ is a $\mathsf{BI}$-algebra.

  The converse is somewhat more involved. First, define a \emph{binary bracketing} (see e.g.,~\cite{Stanley1997}) to be some $\mathcal{B}$ generated by the following grammar:
  \[
    \mathcal{B} ::= \square \mid (\mathcal{B} \mathcal{B})
  \]
  and define the \emph{length} of a binary bracketing to be the number of occurrences of $\square$ it contains, so that for example $(\square(\square \square))$ and $((\square \square)(\square\square))$ are binary bracketings of length $3$ and $4$, respectively.

  Given a binary bracketing $\mathcal{B}$ of length $n$ and a family of morphisms $(\Gamma_i \vdash t_i)_{i \in \{1,\ldots,n\}}$ (where everything is of type $A$), we obtain a morphism $\Gamma_1,\ldots,\Gamma_n \vdash \mathcal{B}(t_1,\ldots,t_n)$, where the term $\mathcal{B}(t_1,\ldots,t_n)$ is obtained by substituting $t_i$ for the $i$th occurrence of $\square$ in $\mathcal{B}$, ordered from left to right. For example, $\square(t) = t$, $(\square(\square \square))(t_1,t_2,t_3) = (t_1(t_2t_3))$ and $((\square \square)(\square \square))(t_1,t_2,t_3,t_4) = ((t_1t_2)(t_3t_4))$.

  Notice that the morphisms of $\mathbf{Id}_\mathsf{r}(A,\application)$ are precisely the morphisms of $\mathcal{M}$ of the form $x_1,\ldots,x_n \vdash \mathcal{B}(x_1,\ldots,x_n)$ where $\mathcal{B}$ is a binary bracketing of length $n$. This means that in order to establish the converse it suffices to show that whenever $(A,\application)$ is a $\mathsf{BI}$-algebra, every such morphism is $(A,\application)$-computable.

  Given a binary bracketing $\mathcal{B}$ of length $n$, for each $1 \leq i \leq n$ we obtain a new binary bracketing $i \triangleright \mathcal{B}$ of length $n+1$ by replacing the $i$th instance of $\square$ in $\mathcal{B}$ by $(\square \square)$. For example, we have $1 \triangleright (\square(\square \square)) = ((\square\square)(\square \square))$, $2 \triangleright (\square(\square \square)) = (\square((\square \square) \square))$, and $3 \triangleright (\square(\square \square)) = (\square(\square(\square \square)))$. Observe that for every binary bracketing $\mathcal{B}$ of length $n+1$ there exists a binary bracketing $\mathcal{B}'$ of length $n$ and $1 \leq i \leq n$ such that $\mathcal{B} = i \triangleright \mathcal{B}'$.

  We proceed to show that if $(A,\application)$ is a $\mathsf{BI}$-algebra then for every binary bracketing $\mathcal{B}$ of length $n$ the morphism $x_1,\ldots,x_n \vdash \mathcal{B}(x_1,\ldots,x_n)$ is $(A,\application)$-computable. We do this by induction on $n$. The base case is when $n = 1$, in which case $\mathcal{B}$ is $\square$, which is $(A,\application)$-computable as in $x \vdash \square(x) = x = \mathsf{I}x$. For the inductive case, suppose that the claim holds for all binary bracketings of length at most $n$, and let $\mathcal{B}$ be a binary bracketing of length $n+1$. Then for some binary bracketing $\mathcal{B}'$ of length $n$ and $1 \leq i \leq n$ we have $\mathcal{B} = i \triangleright \mathcal{B}'$. Moreover, our inductive hypothesis gives that $x_1,\ldots,x_n \vdash \mathcal{B}'(x_1,\ldots,x_n)$ is $(A,\application)$-computable, which is to say that for some $a \in \mathcal{M}(;A)$ we have $x_1,\ldots,x_n \vdash ax_1\cdots x_n = \mathcal{B}'(x_1,\ldots,x_n)$. Now Lemma~\ref{lem:iterated-b} gives that the morphism corresponding to $\mathcal{B}$ is $(A,\application)$-computable:
  \begin{align*}
    & x_1,\ldots,x_{i-1},y_1,y_2,x_{i+1},\ldots,x_n \vdash
      \mathsf{B}^{i-1}\mathsf{B}ax_1\cdots x_{i-1}y_1y_2x_{i+1} \cdots x_n
    \\&= \mathsf{B}(ax_1\cdots x_{i-1})y_1y_2 x_{i+1}\cdots x_n
    = ax_1\cdots x_{i-1} (y_1y_2) x_{i+1}\cdots x_n
    \\&= \mathcal{B}'(x_1,\ldots,x_{i-1},(y_1y_2),x_{i+1},\ldots,x_n)
    \\&= (i \triangleright \mathcal{B}')(x_1,\ldots,x_{i-1},y_1,y_2,x_{i+1},\ldots,x_n)
    \\&= \mathcal{B}(x_1,\ldots,x_{i-1},y_1,y_2,x_{i+1},\ldots,x_n)
  \end{align*}
  The claim follows.
\end{proof}

We proceed to consider the other combinators:
\begin{definition}
  Let $\mathfrak{S}$ be a faithful cartesian club, $\mathcal{M}$ be an $\mathfrak{S}$-multicategory, and $(A,\bullet)$ be an applicative system in $\mathcal{M}$.
  \begin{itemize}
  \item If $\mathfrak{S}$ contains the transpositions, we define a \emph{$\mathsf{C}$ combinator} for $(A,\bullet)$ to be a morphism $\mathsf{C} \in \mathcal{M}(;A)$ such that $\application^3 \circ (\mathsf{C},1_A,1_A,1_A) = [\application^2]\tau_2^3$, or in the sequent calculus notation $x_1,x_2,x_3 \vdash \mathsf{C}x_1x_2x_3 = x_1x_3x_2$.
  \item If $\mathfrak{S}$ contains the degeneracy maps, we define a \emph{$\mathsf{W}$ combinator} for $(A,\bullet)$ to be a morphism $\mathsf{W} \in \mathcal{M}(;A)$ such that $\bullet^2 \circ (\mathsf{W},1_A,1_A) = [\bullet^3]\sigma_2^2$, or in the sequent calculus notation $x_1,x_2 \vdash \mathsf{W}x_1x_2 = x_1x_2x_2$.
  \item If $\mathfrak{S}$ contains the face maps, we define a \emph{$\mathsf{K}$ combinator} for $(A,\bullet)$ to be a morphism $\mathsf{K} \in \mathcal{M}(;A)$ such that $\application^2 \circ (\mathsf{K},1_A,1_A) = [1_A]\delta_2^2$, or in the sequent calculus notation $x_1,x_2 \vdash \mathsf{K}x_1x_2 = x_1$.
  \end{itemize}
\end{definition}
As in the classical case, we say that an applicative system is, for example, a $\mathsf{BCI}$-algebra in case it has a $\mathsf{B}$, $\mathsf{C}$, and $\mathsf{I}$ combinator. We prove a technical lemma:
\begin{lemma}\label{lem:other-combinators}
  Let $\mathfrak{S}$ be a faithful cartesian club, $\mathcal{M}$ be an $\mathfrak{S}$-multicategory, and $(A,\bullet)$ be $\mathsf{BI}$-algebra in $\mathcal{M}$. Then we have:
  \begin{enumerate}
  \item If $\mathfrak{S}$ contains the transpositions and $(A,\bullet)$ has a $\mathsf{C}$ combinator, then whenever $f \in \mathcal{M}(A^n;A)$ is $(A,\bullet)$-computable so is $[f]\tau_i^n$ for $1 \leq i < n$. 
  \item If $\mathfrak{S}$ contains the degeneracy maps and $(A,\bullet)$ has a $\mathsf{W}$ combinator, then whenever $f \in \mathcal{M}(A^{n+1};A)$ is $(A,\bullet)$-computable so is $[f]\sigma_i^n$ for $1 \leq i \leq n$.
  \item If $\mathfrak{S}$ contains the face maps and $(A,\bullet)$ has a $\mathsf{K}$ combinator, then whenever $f \in \mathcal{M}(A^{n-1};A)$ is $(A,\bullet)$-computable so is $[f]\delta_i^n$ for $1 \leq i \leq n$.
  \end{enumerate}
\end{lemma}
  \begin{proof}
    \begin{enumerate}
    \item  Since $f \in \mathcal{M}(A^n;A)$ is $(A,\bullet)$-computable we have $a \in \mathcal{M}(;A)$ such that:
      \begin{align*}
        x_1,\ldots,x_i,x_{i+1},\ldots,x_n \vdash ax_1\cdots x_ix_{i+1}\cdots x_n = f(x_1,\ldots,x_i,x_{i+1},\ldots,x_n)
      \end{align*}
      Then using Lemma~\ref{lem:iterated-b} we have:
      \begin{align*}
        & x_1,\ldots,x_i,x_{i+1},\ldots,x_n
          \vdash \mathsf{B}^{i-1}\mathsf{C}ax_1 \cdots x_i x_{i+1} \cdots x_n
        \\&= \mathsf{C}(ax_1 \cdots) x_i x_{i+1} \cdots x_n
        = ax_1 \cdots x_{i+1}x_i \cdots x_n
        \\&= f(x_1,\ldots,x_{i+1},x_i,\ldots,x_n)
        = [f]\tau^n_i(x_1,\ldots,x_i,x_{i+1},\ldots,x_n)
      \end{align*}
      and the claim follows.
    \item Since $f \in \mathcal{M}(A^{n+1};A)$ is $(A,\bullet)$-computable we have $a \in \mathcal{M}(;A)$ such that:
      \[
        x_1,\ldots,x_{n+1} \vdash ax_1 \cdots x_{n+1} = f(x_1, \ldots, x_{n+1})
      \]
      Then using Lemma~\ref{lem:iterated-b} we have:
      \begin{align*}
        & x_1,\ldots,x_i,\ldots,x_n 
          \vdash \mathsf{B}^{i-1}\mathsf{W}ax_1 \cdots x_i \cdots x_n
          = \mathsf{W}(ax_1 \cdots)x_i \cdots x_n
        \\&= ax_1\cdots x_i x_i \cdots x_n
        = f(x_1,\ldots,x_i,x_i,\ldots,x_n)
        = [f]\sigma^n_i(x_1,\ldots,x_i,\ldots,x_n)
      \end{align*}
      and the claim follows.
    \item Since $f \in \mathcal{M}(A^{n-1};A)$ is $(A,\bullet)$-computable we have $a \in \mathcal{M}(;A)$ such that:
      \[
        x_1,\ldots,x_n \vdash ax_1\cdots x_{n-1} = f(x_1,\ldots,x_{n-1})
      \]
      Then using Lemma~\ref{lem:iterated-b} we have:
  \begin{align*}
    &x_1,\ldots,x_{i-1},x_i,x_{i+1},\ldots,x_n \vdash
    \mathsf{B}^{i-1}\mathsf{K}ax_1\cdots x_{i-1}x_ix_{i+1} \cdots x_n
    \\&= \mathsf{K}(ax_1\cdots x_{i-1})x_ix_{i+1}\cdots x_n
    = ax_1\cdots x_{i-1}x_{i+1}\cdots x_n
    \\&= f(x_1,\ldots,x_{i-1},x_{i+1},\ldots,x_n)
    = [f]\delta_i^n(x_1,\ldots,x_{i-1},x_i,x_{i+1},\ldots,x_n)
  \end{align*}
  and the claim follows.
    \end{enumerate}
  \end{proof}

The other results concerning weak combinatory completeness follow easily:
\begin{theorem}\label{thm:structured-combinatory-complete}
  Let $\mathfrak{S}$ be a faithful cartesian club, $\mathcal{M}$ be an $\mathfrak{S}$-multicategory, and $(A,\bullet)$ be an applicative system in $\mathcal{M}$. Then we have:
  \begin{enumerate}
  \item If $\mathfrak{S}$ contains the bijections, then $(A,\bullet)$ is weakly $\mathbf{Bij}$-combinatory complete if and only if it is a $\mathsf{BCI}$-algebra.
  \item If $\mathfrak{S}$ contains the monotone injections, then $(A,\bullet)$ is weakly $\mathbf{MInj}$-\hspace{0pt}combinatory complete if and only if it is a $\mathsf{BKI}$-algebra.
  \item If $\mathfrak{S}$ contains the injections, then $(A,\bullet)$ is weakly $\mathbf{Inj}$-combinatory complete if and only if it is a $\mathsf{BCKI}$-algebra.
  \item If $\mathfrak{S}$ contains the surjections, then $(A,\bullet)$ is weakly $\mathbf{Srj}$-combinatory complete if and only if it is a $\mathsf{BCWI}$-algebra.
  \item If $\mathfrak{S}$ is $\mathbf{Fun}$, then $(A,\bullet)$ is weakly $\mathbf{Fun}$-combinatory complete if and only if it is a $\mathsf{BCKWI}$-algebra.
  \end{enumerate}
\end{theorem}
\begin{proof}
  We give the proof for $\mathsf{BCWI}$-algebras. The other cases are similar. Suppose that $(A,\bullet)$ is weakly $\mathbf{Srj}$-combinatory complete. Then $(A,\bullet)$ is weakly $\mathbf{Id}$-combinatory complete, and Theorem~\ref{thm:id-combinatory-complete} gives that it is a $\mathsf{BI}$-algebra. Moreover, $[\application^2]\tau_2^3$ and  $[\bullet^3]\sigma_2^2$ are in $\mathbf{Srj}_\mathsf{r}(A,\bullet)$, and so they are $(A,\bullet)$-computable, which is to say that $(A,\bullet)$ has a $\mathsf{C}$ and $\mathsf{W}$ combinator, and is therefore a $\mathsf{BCWI}$-algebra.

  For the converse, suppose that $(A,\bullet)$ is a $\mathsf{BCWI}$-algebra. We must show that every morphism of $\mathbf{Srj}_\mathsf{r}(A,\bullet)$ is $(A,\bullet)$-computable. Theorem~\ref{thm:id-combinatory-complete} gives that every morphism of $\mathbf{Id}_\mathsf{r}(A,\bullet)$ is $(A,\bullet)$-computable. We know $\mathbf{Srj}$ is generated by transpositions and degeneracy maps, so it suffices to show that the $(A,\bullet)$-computable morphisms are closed under the action of such maps, but this is part of Lemma~\ref{lem:other-combinators}. 
\end{proof}

\section{Combinatory Completeness and Multicategory Structure}\label{sec:multicategory-structure}

In this section we give a notion of $\mathfrak{S}$-combinatory complete applicative system, precisely characterise the difference between $\mathfrak{S}$-combinatory completeness and weak $\mathfrak{S}$-combinatory completeness, and show that a given applicative system is $\mathfrak{S}$-combinatory complete if and only if its computable maps form a sub-$\mathfrak{S}$-multicategory of the ambient one. We begin with a suitable analogue of the polynomials over an applicative system:
\begin{definition}\label{def:s-polynomial}
  Let $\mathfrak{S}$ be a faithful cartesian club, and let $(A,\application)$ be an applicative system in an $\mathfrak{S}$-multicategory $\mathcal{M}$. Let $\mathfrak{S}(A,\bullet)$ be the smallest sub-$\mathfrak{S}$-multicategory of $\mathcal{M}$ containing $\bullet \in \mathcal{M}(A,A;A)$ and all $a \in \mathcal{M}(;A)$. We refer to morphisms of $\mathfrak{S}(A,\bullet)$ as \emph{$\mathfrak{S}$-polynomials over $(A,\bullet)$}.
\end{definition}
The associated notion of combinatory completeness is as follows:
\begin{definition}\label{def:s-combinatory-complete}
  Let $\mathfrak{S}$ be a faithful cartesian club, $\mathcal{M}$ be an $\mathfrak{S}$-multicategory, and $(A,\application)$ be an applicative system in $\mathcal{M}$. We say that $(A,\bullet)$ is \emph{$\mathfrak{S}$ combinatory complete} in case every $\mathfrak{S}$-polynomial over $(A,\bullet)$ is $(A,\application)$-computable.
\end{definition}

Compare the polynomials of Definition~\ref{def:s-polynomial} to the regular polynomials of Definition~\ref{def:regular-s-polynomial}. In particular, notice that every regular polynomial is a polynomial. It follows that $\mathfrak{S}$-combinatory completeness (Definition~\ref{def:s-combinatory-complete}) implies weak $\mathfrak{S}$-combinatory completeness (Definition~\ref{def:weak-s-combinatory-complete}). The difference between $\mathfrak{S}$-combinatory completeness and weak $\mathfrak{S}$-combinatory completeness is captured by the following property:  
\begin{definition}[After Tomita~\cite{Tomita2021}]
  Let $\mathcal{M}$ be a multicategory, and let $(A,\bullet)$ be an applicative system in $\mathcal{M}$. We say that $(A,\bullet)$ is \emph{flipped} in case for all $a \in \mathcal{M}$ there exists $h \in \mathcal{M}$ such that $\bullet \circ (h,1_A) = \bullet \circ (1_A,a)$, or equivalently $x \vdash hx = xa$. We will write $a^\bullet$ to denote some such $h$ when it exists.
\end{definition}

In particular, a flipped $\mathsf{BI}$-algebra is precisely a $\mathsf{BI}^\bullet$-algebra in the sense of Tomita~\cite{Tomita2021}. Notice that any $\mathsf{BCI}$-algebra is necessarily flipped:
\begin{lemma}\label{lem:bci-flipped}
  Let $\mathcal{M}$ be a $\mathbf{Bij}$-multicategory, and let $(A,\bullet)$ be a $\mathsf{BCI}$-algebra in $\mathcal{M}$. Then $(A,\bullet)$ is flipped.
\end{lemma}
\begin{proof}
  Let $a \in \mathcal{M}(;A)$. Define $a^\bullet \in \mathcal{M}(;A)$ by $a^\bullet = \bullet^2 \circ (\mathsf{C},\mathsf{I},a)$. Then we have $\bullet \circ (a^\bullet,1_A) = (1_A,a)$ as in:
  \begin{mathpar}
    x \vdash a^\bullet x
    = CIax
    = Ixa
    = xa
  \end{mathpar}
  and the claim follows.
\end{proof}

We require a technical lemma:
\begin{lemma}\label{lem:flip-insert}
  Let $\mathcal{M}$ be a multicategory, and let $(A,\bullet)$ be a flipped $\mathsf{BI}$-algebra in $\mathcal{M}$. If $n \geq 1$ and $x_1,\ldots,x_n \vdash f(x_1,\ldots,x_n) \in \mathcal{M}(A^n;A)$ is $(A,\bullet)$-computable then so is the morphism:
  \[
    x_1,\ldots,x_{i-1},x_{i+1},\ldots,x_n \vdash f(x_1,\ldots,x_{i-1},b,x_{i+1},\ldots,x_n) \in \mathcal{M}(A^{n-1};A)
  \]
  for any $b \in \mathcal{M}(;A)$ and $1 \leq i \leq n$. 
\end{lemma}
\begin{proof}
  Given $a \in \mathcal{M}$ such that $x_1,\ldots,x_n \vdash ax_1 \cdots x_n = f(x_1,\ldots,x_n)$ and an arbitrary $b \in \mathcal{M}(;A)$ we use Lemma~\ref{lem:iterated-b} to obtain:
  \begin{align*}
    & x_1,\ldots,x_{i-1},x_{i+1},\ldots,x_n \vdash \mathsf{B}^ib^\bullet ax_1 \cdots x_{i-1}x_{i+1}\cdots x_n
    \\&= b^\bullet (ax_1 \cdots x_{i-1})x_{i+1} \cdots x_n
    = (ax_1 \cdots x_{i-1})bx_{i+1} \cdots x_n
    \\&= ax_1 \cdots x_{i-1}bx_{i+1} \cdots x_n
    = f(x_1,\ldots,x_{i-1},b,x_{i+1},\ldots,x_n)
  \end{align*}
  and the claim follows.
\end{proof}

Now, we have:
\begin{theorem}\label{thm:flipped-complete}
  Let $\mathfrak{S}$ be a faithful cartesian club, $\mathcal{M}$ be an $\mathfrak{S}$-multicategory, and $(A,\bullet)$ be a weakly $\mathfrak{S}$-combinatory complete applicative system in $\mathcal{M}$. Then $(A,\bullet)$ is flipped if and only if it is $\mathfrak{S}$-combinatory complete.
\end{theorem}
\begin{proof}
  If $(A,\bullet)$ is $\mathfrak{S}$-combinatory complete then for all $a \in \mathcal{M}(;A)$ the $\mathfrak{S}$-polynomial $x_1 \vdash x_1a$ is $(A,\bullet)$-computable, which means that there exists $a^\bullet \in \mathcal{M}(;A)$ such that $x_1 \vdash a^\bullet x_1 = x_1 a$. It follows that $\vdash a^\bullet b = b a$ for all $b \in \mathcal{M}(;A)$, which means that $(A,\bullet)$ is flipped. Conversely, if $(A,\bullet)$ is flipped, then, for any polynomial
  \begin{align}\label{eqn:poly}
      x_1,\ldots,x_n \vdash f(x_1,\ldots,x_n)
  \end{align}
  there exists $k\in\mathbb{N}$, a regular polynomial
  \begin{equation}\label{eqn:reg-poly}
      \begin{aligned}
        &x_1,\ldots,x_{i_1},y_1,x_{i_1 + 1},\ldots,x_{i_k},y_k,x_{i_k + 1},\ldots,x_n \vdash
        \\
        &f'(x_1,\ldots,x_{i_1},y_1,x_{i_1 + 1},\ldots,x_{i_k},y_k,x_{i_k + 1},\ldots,x_n)
      \end{aligned}
  \end{equation}
  and $a_1,\ldots,a_k \in \mathcal{M}(;A)$
  such that
  \begin{align*}
      x_1,\ldots,x_n \vdash f'(x_1,\ldots,x_{i_1},a_1,x_{i_1 + 1},\ldots,x_{i_k},a_k,x_{i_k + 1},\ldots,x_n) = f(x_1,\ldots,x_n)
  \end{align*}
  So by repeated application of Lemma~\ref{lem:flip-insert}, computability of (\ref{eqn:poly}) follows from computability of (\ref{eqn:reg-poly}). 
\end{proof}

Our next goal will be to establish that $\mathfrak{S}$-combinatory completeness corresponds to $\mathfrak{S}$-\hspace{0pt}multicategory structure on the category of computable morphisms. We begin with $\mathsf{Id}$-combinatory completeness, showing that the computable maps of any flipped $\mathsf{BI}$-algebra form a multicategory:
\begin{lemma}
  Let $\mathcal{M}$ be a multicategory, and let $(A,\bullet)$ be a flipped $\mathsf{BI}$-algebra in $\mathcal{M}$. Then the $(A,\bullet)$-computable maps form a sub-multicategory of $\mathcal{M}$.
\end{lemma}
\begin{proof}
  The identity morphism $1_A$ is $(A,\bullet)$-computable via the $\mathsf{I}$-combinator. To see that the $(A,\bullet)$-computable morphisms are closed under composition, suppose $g \in \mathcal{M}(A^n;A)$ and $(f_i \in \mathcal{M}(A^{m_i};A))_{i=1}^n$ are $(A,\bullet)$-computable. That is, we have $b,a_1,\ldots,a_n \in \mathcal{M}(;A)$ such that $x_1,\ldots,x_n \vdash bx_1 \cdots x_n = g(x_1,\ldots,x_n)$ and $x_1^i,\ldots,x_{m_i}^i \vdash a_ix_1^i \cdots x_{m_i}^i = f(x^i_1,\ldots,x^i_{m_i})$ for all $1 \leq i \leq n$. Since $(A,\bullet)$ is a flipped $\mathsf{BI}$-algebra we know that it is $\mathbf{Id}$-combinatory complete, which in particular means that the following $\mathbf{Id}$-polynomial is $(A,\bullet)$-computable:
  \begin{equation*}    x^1_1,\ldots,x^1_{m_1},\ldots,x^n_1,\ldots,x^n_{m_n} \vdash b(a_1x^1_{m_1} \cdots x^1_{m_1})\cdots (a_n x^n_1 \cdots x^n_{m_n})
  \end{equation*}
  This is to say that for some $c \in \mathcal{M}(;A)$ we have:
  \begin{align*}
    & x^1_1,\ldots,x^1_{m_1},\ldots,x^n_1,\ldots,x^n_{m_n} \vdash cx_1^1  \cdots x^1_{m_1} \cdots x^n_1 \cdots x^n_{m_n}
    \\&= b(a_1x^1_{m_1} \cdots x^1_{m_1})\cdots (a_n x^n_1 \cdots x^n_{m_n})
    \\&= g(f_1(x^1_1,\ldots,x^1_{m_1}),\ldots, f_n(x^n_1,\ldots,x^n_{m_n}))
    \\&= (g\circ (f_1,\ldots,f_n))( x_1^1,\ldots, x^1_{m_1},\ldots,x^n_1,\ldots, x^n_{m_n})
  \end{align*}
\end{proof}
and so $g \circ (f_1,\ldots,f_n)$ is $(A,\bullet)$-computable. The claim follows. 

Next, we consider arbitrary faithful cartesian clubs $\mathfrak{S}$:
\begin{lemma}
  Let $\mathfrak{S}$ be a faithful cartesian club, $\mathcal{M}$ be an $\mathfrak{S}$-multicategory, and $(A,\bullet)$ be an $\mathfrak{S}$-combinatory complete applicative system in $\mathcal{M}$. Then the $(A,\bullet)$-computable morphisms of $\mathcal{M}$ form a sub-$\mathfrak{S}$-multicategory of $\mathcal{M}$.
\end{lemma}
\begin{proof}
  For any faithful cartesian club $\mathfrak{S}$ the $\mathfrak{S}$-polynomials contain the $\mathbf{Id}$-polynomials, so we know that $(A,\bullet)$ is $\mathbf{Id}$-combinatory complete, which means that it is a flipped $\mathsf{BI}$-algebra, which means that the $(A,\bullet)$-computable maps form a sub-multicategory of $\mathcal{M}$. It remains to show that the $(A,\bullet)$-computable maps are closed under the action of $\mathfrak{S}$ on $\mathcal{M}$. To that end, suppose $f \in \mathcal{M}(A^m;A)$ is $(A,\bullet)$-computable and that $\mathbf{a} : m \to n \in \mathfrak{S}$. We must show that $x_1,\ldots,x_n \vdash [f]\mathbf{a} \in \mathcal{M}(A^n;A)$ is $(A,\bullet)$-computable. Since $f$ is $(A,\bullet)$-computable there is some $a \in \mathcal{M}(;A)$ such that $x_1,\ldots,x_m \vdash ax_1 \cdots x_m = f(x_1,\ldots,x_m)$. Now, $x_1,\ldots,x_m \vdash ax_1 \cdots x_m$ is an $\mathbf{Id}$-polynomial over $(A,\bullet)$, and so $x_1,\ldots,x_n \vdash [ax_1\cdots x_m]\mathbf{a}$ is an $\mathfrak{S}$-polynomial over $(A,\bullet)$, but then since $(A,\bullet)$ is $\mathfrak{S}$-combinatory complete we know that it is $(A,\bullet)$-computable. But this is $[f]\mathbf{a}$ as in:
  \begin{align*}
    & x_1,\ldots,x_n \vdash [ax_1 \cdots x_m]\mathbf{a}
      = [f(x_1,\ldots,x_m)]\mathbf{a}
      = f(x_{\mathbf{a}(1)},\ldots,x_{\mathbf{a}(m)})
      \\&= ([f]\mathbf{a})(x_1,\ldots,x_n)
  \end{align*}
  so $[f]\mathbf{a}$ is $(A,\bullet)$-computable, and the claim follows.
\end{proof}

The converse is relatively straightforward:
\begin{lemma}
  Let $\mathfrak{S}$ be a faithful cartesian club, $\mathcal{M}$ be an $\mathfrak{S}$-multicategory, and $(A,\bullet)$ be an applicative system in $\mathcal{M}$. Suppose that the $(A,\bullet)$-computable morphisms form a sub-$\mathfrak{S}$ multicategory of $\mathcal{M}$. Then $(A,\bullet)$ is $\mathfrak{S}$-combinatory complete. 
\end{lemma}
\begin{proof}
  We know that $1_A$ is computable, which is to say that for some $I \in \mathcal{M}(;A)$ we have $x_1 \vdash Ix_1 = x_1$, and so $(A,\bullet)$ has an $\mathsf{I}$-combinator. Now we have $x_1,x_2 \vdash Ix_1x_2 = x_1x_2$ and so $\bullet \in \mathcal{M}(A,A;A)$ is $(A,\bullet)$-computable. But then since the $(A,\bullet)$-computable morphisms form a sub-multicategory of $\mathcal{M}$ we have that $\bullet \circ (1_A,\bullet) \in \mathcal{M}(A^3;A)$ is $(A,\bullet)$-computable, which is to say that $(A,\bullet)$ has a $\mathsf{B}$-combinator. Similarly, any $a \in \mathcal{M}(;A)$ is $(A,\bullet)$-computable which means that so is $\bullet \circ (1_A,a)$. That is, for each $a \in \mathcal{M}(;A)$ there is some $a^\bullet \in \mathcal{M}(;A)$ such that $x_1 \vdash a^\bullet x_1 = x_1 a$. Thus, $(A,\bullet)$ is a flipped $\mathsf{BI}$-algebra, which is to say that it is $\mathbf{Id}$-combinatory complete. We must show that every $\mathfrak{S}$-polynomial over $(A,\bullet)$ is $(A,\bullet)$-computable, but this follows from the $(A,\bullet)$-computability of the $\mathbf{Id}$-polynomials together with the fact that the $(A,\bullet)$-computable morphisms are closed under the action of $\mathfrak{S}$ on $\mathcal{M}$. Thus, $(A,\bullet)$ is $\mathfrak{S}$-combinatory complete. 
\end{proof}

We have now shown:
\begin{theorem}\label{thm:combinatory-complete-multicategory}
  Let $\mathfrak{S}$ be a faithful cartesian club, $\mathcal{M}$ be an $\mathfrak{S}$-multicategory, and $(A,\bullet)$ be an applicative system in $\mathcal{M}$. Then $(A,\bullet)$ is $\mathfrak{S}$-combinatory complete if and only if the $(A,\bullet)$-computable morphisms form a sub-$\mathfrak{S}$-multicategory of $\mathcal{M}$.
\end{theorem}

For example, we have that for an applicative system $(A,\bullet)$ in a $\mathsf{Bij}$-multicategory $\mathcal{M}$, the following are equivalent:
\begin{enumerate}
\item $(A,\bullet)$ is a $\mathsf{BCI}$-algebra.
\item $(A,\bullet)$ is $\mathbf{Bij}$-combinatory complete.
\item The $(A,\bullet)$-computable maps form a sub-$\mathbf{Bij}$-multicategory of $\mathcal{M}$.
\end{enumerate}

So too for the other rows of Figure~\ref{fig:results-summary}, with the caveat that we must work with flipped applicative systems in those cases that are not flipped already (being the case of $\mathsf{BI}$-algebras and $\mathbf{Id}$-combinatory completeness and $\mathsf{BKI}$-algebras and $\mathbf{Minj}$-combinatory completeness).

\section{Closed Multicategories and Extensionality}\label{sec:closed-extensional}

Given the connection between combinatory completeness and multicategory structure established in the previous section, it is natural to wonder how additional structure on the multicategory of computable maps is reflected in terms of the underlying applicative system. In this section we answer this question for \emph{closed} multicategory structure. In particular, we show that it corresponds to the underlying applicative system being \emph{extensional} in an appropriate sense.

We begin by recalling the definition of closed multicategory:
\begin{definition}[\cite{Manzyuk2012}]
  A \emph{closed multicategory} consists of a multicategory $\mathcal{M}$ together with an object $[A,B] \in \mathcal{M}_0$ and morphism $\mathsf{ev}_{A,B} \in \mathcal{M}([A,B],A;B)$ for each $A,B \in \mathcal{M}_0$ such that for every $\Gamma \in \mathcal{M}_0^*$ and every $f \in \mathcal{M}(\Gamma,A;B)$ there exists a unique $h \in \mathcal{M}(\Gamma;[A,B])$ such that $\mathsf{ev}_{A,B} \circ (h,1_A) = f$.
\end{definition}
    
Next, we recall the notion of extensionality. Classically, an applicative system $(A,\bullet)$ is said to be \emph{extensional} in case for all $a,b \in A$, if $ax = bx$ for all $x\in A$ then $a=b$. This definition lifts easily to our setting:

\begin{definition}
  Let $\mathcal{M}$ be a multicategory, and let $(A,\bullet)$ be an applicative system in $\mathcal{M}$. Say that $(A,\bullet)$ is \emph{extensional} in case for all $a,b \in \mathcal{M}(;A)$, if $\bullet \circ (a,1_A) = \bullet \circ (b,1_A)$ then $a = b$.
\end{definition}

However, we need a slightly stronger notion of extensionality:
\begin{definition}
  Let $\mathcal{M}$ be a multicategory, and let $(A,\bullet)$ be an applicative system in $\mathcal{M}$. Say that $(A,\bullet)$ is \emph{multi-extensional} in case for all $n \in \mathbb{N}$ and $a,b \in \mathcal{M}(;A)$, if $\bullet^n \circ (a,1_A^n) = \bullet^n \circ (b,1_A^n)$ then $a = b$. 
\end{definition}
An equivalent formulation of multi-extensionality is that for all $n \in \mathbb{N}$ and $a,b \in \mathcal{M}(;A)$, if $\bullet^{n+1} \circ (a,1^{n+1}_A) = \bullet^{n+1} \circ (b,1^{n+1}_A)$ then $\bullet^n \circ (a,1_A^n) = \bullet^n \circ (b,1_A^n)$. Notice that multi-extensionality implies extensionality.

While multi-extensionality might seem to be a much stronger property than mere extensionality, they are equivalent in the classical setting. In particular, consider the following property:
\begin{definition}
  Say that a multicategory $\mathcal{M}$ has \emph{enough points} in case for all $f,g \in \mathcal{M}(A_1,\ldots,A_n ; B)$, if $f \circ (x_1,\ldots,x_n) = g \circ (x_1,\ldots,x_n)$ for all $(x_i \in \mathcal{M}(;A_i))_{i=1}^n$ then $f = g$.
\end{definition}
In particular, $\mathsf{Set}$ has enough points. We have:
\begin{lemma}
  Let $\mathcal{M}$ be a multicategory, and let $(A,\bullet)$ be an applicative system in $\mathcal{M}$. If $\mathcal{M}$ has enough points, then $(A,\bullet)$ is extensional if and only if it is multi-extensional.
\end{lemma}
\begin{proof}
  Since multi-extensionality always implies extensionality, it suffices to show that if $(A,\bullet)$ is extensional then it is multi-extensional, under the assumption that $\mathcal{M}$ has enough points. Specifically, we must show that $\forall a,b \in \mathcal{M}(;A)$ and $\forall n \in \mathbb{N}$, $\bullet^n \circ (a,1_A^n) = \bullet^n \circ (b,1_A^n)$ implies $a = b$. We proceed by induction on $n$. If $n = 0$ then $a = b$ immediately. If $n = k+1$ and $\bullet^k \circ (a,1_A^k) = \bullet^k \circ (b,1_A^k)$ implies $a = b$, suppose $\bullet^{k+1} \circ (a,1_A^{k+1}) = \bullet^{k+1} \circ (b,1_A^{k+1})$. Then for all $x_1,\ldots,x_k \in \mathcal{M}(;A)$ we have:
  \begin{align*}
    & \bullet \circ (\bullet^k \circ (a,x_1,\ldots,x_k),1_A)
      = \bullet \circ (\bullet^k \circ (a,1_A^k),1_A) \circ (x_1,\ldots,x_k,1_A)
      \\& = \bullet^{k+1} \circ (a,1_A^{k+1}) \circ (x_1,\ldots,x_k,1_A)
      = \bullet^{k+1} \circ (b,1_A^{k+1}) \circ (x_1,\ldots,x_k,1_A)
      \\& = \bullet \circ (\bullet^k \circ (b,1_A^k),1_A) \circ (x_1,\ldots,x_k,1_A)
      = \bullet \circ (\bullet^k \circ (b,x_1,\ldots,x_k),1_A)
  \end{align*}
  now extensionality of $(A,\bullet)$ gives $\bullet^k \circ (a,x_1,\ldots,x_k) = \bullet^k \circ (b,x_1,\ldots,x_k)$, from which we use that $\mathcal{M}$ has enough points to obtain $\bullet^k \circ (a,1_A^k) = \bullet^k \circ (b,1_A^k)$, which gives $a = b$ by our inductive hypothesis. Thus, $(A,\bullet)$ is multi-extensional, and the claim follows.
\end{proof}

Multi-extensionality is the property that one actually wants. In particular, we have:
\begin{lemma}
  Let $\mathfrak{S}$ be a faithful cartesian club, $\mathcal{M}$ be an $\mathfrak{S}$-\hspace{0pt}multicategory, and $(A,\bullet)$ be an $\mathfrak{S}$-combinatory complete applicative system in $\mathcal{M}$. If $(A,\bullet)$ is multi-extensional, then the sub-\hspace{0pt}multicategory of $(A,\bullet)$-computable morphisms is a closed multicategory with $[A,A] = A$ and $\mathsf{ev}_{A,A} = \bullet$.
\end{lemma}
\begin{proof}
  For any $(A,\bullet)$-computable $f \in \mathcal{M}(A^n,A;A)$ we have $a \in \mathcal{M}(;A)$ such that $\bullet^{n+1} \circ (a,1^{n+1}_A) = f$. Now since $(A,\bullet)$ is $\mathfrak{S}$ combinatory complete we know that the $\mathbf{Id}$-polynomial $\bullet^{n} \circ (a,1^n_A) \in \mathcal{M}(A^n;A)$ is $(A,\bullet)$-computable, and moreover we have:
  \begin{align*}
    & \mathsf{ev}_{A,A} \circ (\bullet^n \circ (a,1_A^n),1_A)
      = \bullet \circ (\bullet^n \circ (a,1_A^n),1_A)
      = \bullet^{n+1} \circ (a,1_A^{n+1})
      = f
  \end{align*}
  We show that $\bullet^n \circ (a,1_A^n)$ is the unique $(A,\bullet)$-computable morphism with this property. To that end, suppose that we have some $(A,\bullet)$-computable $h \in \mathcal{M}(A^n;A)$ with $\mathsf{ev}_{A,A} \circ (h,1_A) = f$. Then since $h$ is $(A,\bullet)$-computable we have $b \in \mathcal{M}(;A)$ such that $\bullet^n \circ (b,1_A^n) = h$, but then we have:
  \begin{align*}
    & \bullet^{n+1} \circ (b,1_A^{n+1})
      = \bullet \circ (\bullet^n \circ  (b,1_A^n),1_A)
      = \bullet \circ (h,1_A)
     \\& = \mathsf{ev}_{A,A} \circ (h,1_A)
      = f
      = \bullet^{n+1} \circ (a,1_A^{n+1})
  \end{align*}
  and so since $(A,\bullet)$ is multi-extensional we have $a = b$. Now we have $h = \bullet^n \circ (b,1_A^n) = \bullet^n \circ (a,1_A^n)$, and the claim follows.
\end{proof}
Conversely:
\begin{lemma}
  Let $\mathfrak{S}$ be a faithful cartesian club, $\mathcal{M}$ be an $\mathfrak{S}$-multicategory, and $(A,\bullet)$ be an $\mathfrak{S}$-combinatory complete applicative system in $\mathcal{M}$. Suppose the sub-multicategory of $(A,\bullet)$-computable morphisms is closed with $[A,A] = A$ and $\mathsf{ev}_{A,A} = \bullet$. Then $(A,\bullet)$ is multi-extensional.
\end{lemma}
\begin{proof}
  We must show that for all $n \in \mathbb{N}$ and all $a,b \in \mathcal{M}(;A)$, $\bullet^n \circ (a,1_A^n) = \bullet^n \circ (b,1_A^n)$ implies $a = b$. We proceed by induction on $n$. For the base case, $\bullet^0 \circ (a,1_A^0) = \bullet^0 \circ (b,1_A^0)$ is precisely the statement that $a = b$. For the inductive case, suppose that $\bullet^n \circ (a,1_A^n) = \bullet^n \circ (b,1_A^n)$ implies $a = b$, and that we have $\bullet^{n+1} \circ (a,1_A^{n+1}) = \bullet^{n+1} \circ (b,1^{n+1}_A)$. Then we have:
  \begin{align*}
    & \mathsf{ev}_{A,A} \circ (\bullet^n \circ (a,1_A^n),1_A)
      = \bullet \circ (\bullet^n \circ (a,1_A^n),1_A)
      = \bullet^{n+1} \circ (a,1_A^{n+1})
      \\& = \bullet^{n+1} \circ (b,1_A^{n+1})
    = \bullet \circ (\bullet^n \circ (b,1_A^n),1_A)
    = \mathsf{ev}_{A,A} \circ (\bullet^n \circ (b,1_A^n),1_A)
  \end{align*}
  But then by the universal property of the closed multicategory structure we have $\bullet^n \circ (a,1_A^n) = \bullet^n \circ (b,1_A^n)$, which gives $a = b$. The claim follows by induction.
\end{proof}

We record:
\begin{theorem}\label{thm:closed-extensional}
  Let $\mathfrak{S}$ be a faithful cartesian club, $\mathcal{M}$ be an $\mathfrak{S}$-multicategory, and $(A,\bullet)$ be an $\mathfrak{S}$-combinatory complete applicative system in $\mathcal{M}$. Then $(A,\bullet)$ is multi-extensional if and only if the sub-multicategory of $(A,\bullet)$-computable morphisms is closed with $[A,A] = A$ and $\mathsf{ev}_{A,A} = \bullet$.
\end{theorem}

\section{Weakly-Closed Operads}\label{sec:weakly-closed-operads}

In this section we give a simple characterisation of those multicategories that arise as the computable maps of some combinatory complete applicative system, which will turn out to be the \emph{weakly-closed operads}. First, recall:
\begin{definition}\label{def:operad}
  An \emph{operad} is a multicategory with exactly one object.
\end{definition}
When working with an operad $\mathcal{M}$ we will write $*_\mathcal{M}$ or simply $*$ to denote the unique object of $\mathcal{M}$. For hom-sets, we write $\mathcal{M}(n) = \mathcal{M}(*^n;*)$.

Next, we define weakly-closed multicategories, which are a weakening of the notion of closed multicategory:
\begin{definition}\label{def:weakly-closed-multicategory}
  A \emph{weakly-closed multicategory} consists of a multicategory $\mathcal{M}$ together with an object $[A,B] \in \mathcal{M}_0$ and morphism $\mathsf{ev}_{A,B} \in \mathcal{M}([A,B],A;B)$ for each $A,B \in \mathcal{M}_0$ such that for every $\Gamma \in \mathcal{M}_0^*$ and every $f \in \mathcal{M}(\Gamma,A;B)$ there exists $h \in \mathcal{M}(\Gamma;[A,B])$ such that $\mathsf{ev}_{A,B} \circ (h,1_A) = f$.
\end{definition}
The only difference between a weakly-closed multicategory and a closed multicategory is that the representative of a given morphism need not be unique. The associated notion of structure-preserving functor is straightforward:
\begin{definition}
  Let $\mathcal{M}$ and $\mathcal{N}$ be weakly-closed multicategories. Say that a multifunctor $F : \mathcal{M} \to \mathcal{N}$ is itself \emph{weakly-closed} in case $F[A,B] = [FA,FB]$ and $F(\mathsf{ev}_{A,B}) = \mathsf{ev}_{FA,FB}$ for all $A,B \in \mathcal{M}_0$.
\end{definition}

When working with a weakly-closed operad, it is unambiguous to write $\mathsf{ev} = \mathsf{ev}_{*,*}$. Moreover, since $*$ is the only object, we have $[*,*] = *$, which means that $\mathsf{ev} \in \mathcal{M}(2) = \mathcal{M}(*,*;*)$, so that $(*,\mathsf{ev})$ is an applicative system in any weakly-closed operad $\mathcal{M}$. Moreover, the computable morphisms of this applicative system are precisely the morphisms of $\mathcal{M}$: every $(*,\mathsf{ev})$-computable morphism is by definition a morphism of $\mathcal{M}$, and conversely we have:
\begin{lemma}\label{lem:combinatory-computable}
  Let $\mathcal{M}$ be a weakly-closed operad. Then every morphism of $\mathcal{M}$ is $(*,\mathsf{ev})$-computable.
\end{lemma}
\begin{proof}
  We proceed by induction on $n$. If $n = 0$ then any $f \in \mathcal{M}(0) = \mathcal{M}(;*)$ is $(*,\mathsf{ev})$-computable as in $\mathsf{ev}^0 \circ (f,1_*^0) = 1_* \circ (f) = f$. If all morphisms of $\mathcal{M}(n)$ are $(*,\mathsf{ev})$-computable and $f \in \mathcal{M}(n+1)$ then since $\mathcal{M}$ is weakly-closed we know that for some $h \in \mathsf{M}(*^n;[*,*]) = \mathsf{M}(*^n;*) = \mathsf{M}(n)$ we have $\mathsf{ev} \circ (h,1_*) = f$. Now by assumption we have that $h$ is $(*,\mathsf{ev})$-computable, and so there exists some $a \in \mathcal{M}(;*)$ such that $\mathsf{ev}^n \circ (a,1_*^n) = h$. This gives:
  \begin{align*}
    & \mathsf{ev}^{n+1} \circ (a,1^{n+1}_*)
      = \mathsf{ev} \circ (\mathsf{ev}^n,1_*) \circ (a,1_*^n,1_*)
      \\&= \mathsf{ev} \circ (\mathsf{ev}^n \circ (a,1_*^n),1_*)
      = \mathsf{ev} \circ (h,1_A)
      = f
  \end{align*}
  Thus $f$ is $(*,\mathsf{ev})$-computable, and the claim follows by induction.
\end{proof}

This means that such applicative systems are always combinatory complete:
\begin{lemma}\label{lem:m-computable-maps}
  Let $\mathfrak{S}$ be a faithful cartesian club, and let $\mathcal{M}$ be a weakly-closed $\mathfrak{S}$-operad. Then $(*,\mathsf{ev})$ is $\mathfrak{S}$-combinatory complete.
\end{lemma}
\begin{proof}
  Lemma~\ref{lem:combinatory-computable} gives that every morphism of $\mathcal{M}$ is $(*,\mathsf{ev})$-\hspace{0pt}computable, and of course any $(*,\mathsf{ev})$-\hspace{0pt}computable morphism of $\mathcal{M}$ is in $\mathcal{M}$, so in particular the $(*,\mathsf{ev})$-computable morphisms trivially form a sub-$\mathfrak{S}$-multicategory of $\mathcal{M}$. Then Theorem~\ref{thm:combinatory-complete-multicategory} gives that $(*,\mathsf{ev})$ is $\mathfrak{S}$-combinatory complete.
\end{proof}

Thus, weakly-closed operads induce combinatory complete applicative systems. The converse is also true: given any combinatory complete applicative system its computable morphisms form a weakly-closed operad. We proceed to make this statement a formal one. To this end, we organise a given applicative system $(A,\bullet)$ in a multicategory $\mathcal{M}$ as a tuple $(\mathcal{M},A,\bullet)$. We emphasize that such tuples are to be understood as applicative systems \emph{in context}, where the ``context'' is the ambient multicategory $\mathcal{M}$. More formally:

\begin{definition}
  Let $\mathfrak{S}$ be a faithful cartesian club. An \emph{$\mathfrak{S}$-triple} $(\mathcal{M},A,\bullet)$ consists of an $\mathfrak{S}$-\hspace{0pt}multicategory $\mathcal{M}$ together with an $\mathfrak{S}$-combinatory complete applicative system $(A,\bullet)$ in $\mathcal{M}$. A \emph{morphism} of $\mathfrak{S}$-triples $F : (\mathcal{M},A,\bullet) \to (\mathcal{N},B,*)$ consists of an $\mathfrak{S}$-multifunctor $F : \mathcal{M} \to \mathcal{N}$ such that $FA = B$ and $F(\bullet) = *$. Write $\mathsf{T}_\mathfrak{S}$ for the category of $\mathfrak{S}$-triples and their morphisms.
\end{definition}

We are interested in relating these applicative systems in context to weakly-closed operads, with the latter being like an applicative system \emph{standing alone}. More precisely, a weakly-closed operad can be seen as a combinatory complete applicative system in the context of its category of computable morphisms. For the purpose of this comparison, we require categories of weakly-closed operads:
\begin{definition}
  Let $\mathfrak{S}$ be a faithful cartesian club. Define $\mathsf{O}_\mathfrak{S}$ to be the category of weakly-closed $\mathfrak{S}$-operads and weakly-closed $\mathfrak{S}$-multifunctors between them.
\end{definition}

We proceed to construct adjunctions, parameterised by a faithful cartesian club $\mathfrak{S}$, between our categories of applicative systems in context and of weakly-closed operads. To begin, we construct what will eventually be the object mapping of the right adjoint, which is to construct a weakly-closed operad from a combinatory-complete applicative system in context:

\begin{definition}\label{def:r-objects}
  Let $\mathfrak{S}$ be a faithful cartesian club, $\mathcal{M}$ be an $\mathfrak{S}$-multicategory, and $(A,\bullet)$ be an $\mathfrak{S}$-combinatory complete applicative system in $\mathcal{M}$. We define the weakly-closed $\mathfrak{S}$-operad $R(\mathcal{M},A,\bullet)$ to have morphisms $\overline{f} \in R(\mathcal{M},A,\bullet)(n)$ for each $(A,\bullet)$-computable morphism $f \in \mathcal{M}(A^n;A)$. Composition and identities are as in $\mathcal{M}$, as is the $\mathfrak{S}$-multicategory structure. Explicitly, $\overline{f} \circ (\overline{g_1},\ldots,\overline{g_n}) = \overline{f \circ (g_1,\ldots,g_n)}$, $1_{*} = \overline{1_A}$, and $[\,\overline{f}\,]\mathbf{a} = \overline{[f]\mathbf{a}}$.
\end{definition}

We show that this is well-defined:
\begin{lemma}\label{lem:computable-combinatory}
  Let $\mathfrak{S}$ be a faithful cartesian club, $\mathcal{M}$ be an $\mathfrak{S}$-multicategory, and $(A,\bullet)$ be an $\mathfrak{S}$-combinatory complete applicative system in $\mathcal{M}$. Then $R(M,A,\bullet)$ is a weakly-closed $\mathfrak{S}$-operad. 
\end{lemma}
\begin{proof}
  That $R(\mathcal{M},A,\bullet)$ is an $\mathfrak{S}$-operad follows from the fact that the $(A,\bullet)$-computable morphisms form a sub-$\mathfrak{S}$-multicategory of $\mathcal{M}$ (Theorem~\ref{thm:combinatory-complete-multicategory}). For the weakly closed structure, we have no choice but to take $[*,*] = *$, and we define $\mathsf{ev}_{*,*} = \overline{\bullet} \in R(\mathcal{M},A,\bullet)(2)$.

  Suppose that $\overline{f} \in R(\mathcal{M},A,\bullet)(n+1)$. Then for the corresponding $(A,\bullet)$-computable morphism $f \in \mathcal{M}(A^{n+1};A)$ we have some $a \in \mathcal{M}(;A)$ such that $\bullet^{n+1} \circ (a,1^{n+1}_A) = f$. Now since $(A,\bullet)$ is $\mathfrak{S}$ combinatory complete we know that the $\mathbf{Id}$-polynomial $\bullet^{n} \circ (a,1^n_A) \in \mathcal{M}(A^n;A)$ is $(A,\bullet)$-computable, and so we have $\overline{\bullet^{n} \circ (a,1^n_A)} \in R(\mathcal{M},A,\bullet)(n)$ such that:
  \begin{align*}
    & \mathsf{ev}_{*,*} \circ \left(\overline{\bullet^n \circ (a,1_A^n)}\,,1_*\right)
      = \overline{\bullet} \circ \left(\overline{\bullet^n \circ (a,1_A^n)}\,,\overline{1_A}\right)
      \\&= \overline{\bullet \circ (\bullet^n \circ (a,1_A^n), 1_A)}
      = \overline{\bullet^{n+1} \circ (a,1_A^{n+1})}
      = \overline{f}
  \end{align*}
  and it follows that $R(\mathcal{M},A,\bullet)$ is weakly closed. 
\end{proof}

The above mapping extends to a functor as follows:
\begin{lemma}
  For any faithful cartesian club $\mathfrak{S}$ there is a functor $R : \mathsf{T}_\mathfrak{S} \to \mathsf{O}_\mathfrak{S}$ with object mapping as in Definition~\ref{def:r-objects} and with arrow mapping given as follows: if $F : (\mathcal{M},A,\bullet) \to (\mathcal{N},B,\odot)$ in $\mathsf{T}_\mathfrak{S}$ then $R(F) : R(\mathcal{M},A,\bullet) \to  R(\mathcal{N},B,\odot)$ is the weakly-closed $\mathfrak{S}$-multifunctor given by $R(F)(*) = *$ on objects and by $R(F)(\overline{f}) = \overline{F(f)}$ on morphisms. 
\end{lemma}
\begin{proof}
Lemma~\ref{lem:combinatory-computable} tells us that the object mapping is well-defined. We show $R(F)$ is well-defined as a morphism of $\mathsf{O}_\mathfrak{S}$: for composition, we have:
\begin{align*}
  & R(F)(\overline{f} \circ (\overline{g_1},\ldots,\overline{g_n}))
    = R(F)(\overline{f \circ (g_1,\ldots,g_n)})
    = \overline{F(f\circ (g_1,\ldots,g_n))}
  \\&= \overline{F(f) \circ (F(g_1),\ldots,F(g_n))}
  = \overline{F(f)} \circ (\overline{F(g_1)},\ldots,\overline{F(g_n)})
    \\&= R(F)(\overline{f}) \circ (R(F)(\overline{g_1}),\ldots,R(F)(\overline{g_n}))
\end{align*}
for identities, we have:
\begin{align*}
  & R(F)(1_*) = R(F)(\overline{1_A}) = \overline{F(1_A)} = \overline{1_{FA}} = \overline{1_B} = 1_* = 1_{F*}
\end{align*}
and for the action of $\mathfrak{S}$, we have:
\begin{align*}
  & R(F)([\overline{f}]\mathbf{a})
    = R(F)(\overline{[f]\mathbf{a}})
    = \overline{F([f]\mathbf{a})}
    = \overline{[F(f)]\mathbf{a}}
    = [\overline{F(f)}]\mathbf{a}
    = [R(F)(\overline{f})]\mathbf{a}
\end{align*}
it follows that $R(f)$ is an $\mathfrak{S}$-multifunctor. Moreover, we have:
\[
  R(F)([*,*]) = R(F)(*) = * = [*,*] = [R(F)(*),R(F)(*)]
\]
and
\[
  R(F)(\mathsf{ev}) = R(F)(\overline{\bullet}) = \overline{F(\bullet)} = \overline{\odot} = \mathsf{ev}
\]
which means that $R(F)$ is weakly-closed, and is a morphism of $\mathsf{O}_\mathfrak{S}$ as required.

We must also show that $R$ itself is a functor. For identities, we find that $R(1_\mathcal{(M,A,\bullet)})$ is the identity as in:
\[
  R(1_\mathcal{(M,A,\bullet)})(\overline{f}) = R(1_\mathcal{M})(\overline{f}) = \overline{1_{\mathcal{M}}(f)} = \overline{f}
\]
and for composition we have $R(F \circ G) = R(F) \circ R(G)$ as in:
\[
  R(F \circ G)(\overline{f})
  = \overline{F(G(f))}
  = R(F)(\overline{G(f)})
  = R(F)(R(G)(\overline{f}))
  = (R(F) \circ R(G))(\overline{f})
\]
Thus, $R$ is in fact a functor.
\end{proof}

This functor will be the right adjoint of our adjunction. We proceed to construct the left adjoint:
\begin{lemma}
  For any faithful cartesian club $\mathfrak{S}$ there is a functor $L : \mathsf{O}_\mathfrak{S} \to \mathsf{T}_\mathfrak{S}$ with object mapping given by $L(\mathcal{M}) = (\mathcal{M},\mathsf{ev},\bullet)$ and arrow mapping given as follows: if $F : \mathcal{M} \to \mathcal{N}$ in $\mathsf{O}_\mathfrak{S}$ then $L(F) : L(\mathcal{M}) \to L(\mathcal{N})$ by $L(F) = F$. That is, $L(F)$ is the morphism $(\mathcal{M},*_\mathcal{M},\mathsf{ev}) \to (\mathcal{N},*_\mathcal{N},\mathsf{ev})$ of $\mathsf{T}_\mathfrak{S}$ given by the $\mathfrak{S}$-multifunctor $F : \mathcal{M} \to \mathcal{N}$.
\end{lemma}
\begin{proof}
Lemma~\ref{lem:m-computable-maps} gives that the object mapping is well-defined. For $L(F)$ to be  well-defined as a morphism of $\mathsf{T}_\mathfrak{S}$ we require $F(\mathsf{ev}) = \mathsf{ev}$ and $F(*_\mathcal{M}) = *_\mathcal{N}$. The former holds because $F$ is weakly-closed, and the latter is immediate. That $L$ itself preserves composition and identities is immediate.
\end{proof}

We are now ready to construct our adjunction:
\begin{theorem}\label{thm:adjunction}
  Let $\mathfrak{S}$ be a faithful cartesian club. Then there is an adjunction:
  \[
    \begin{tikzcd}
      \mathsf{O}_\mathfrak{S} \ar[r,phantom,"{\scriptstyle \perp}"] \ar[r,"L", shift left=0.5em] & \mathsf{T}_\mathfrak{S} \ar[l,"R", shift left=0.5em]
    \end{tikzcd}
  \]
  Moreover, the unit of this adjunction is an isomorphism.
\end{theorem}
\begin{proof}
  The components $\eta_\mathcal{M} : \mathcal{M} \to R(\mathcal{M},*,\mathsf{ev})$ of the unit $\eta : 1 \to R \circ L$ are defined by $\eta_\mathcal{M}(*) = *$ on objects and by $\eta_\mathcal{M}(f) = \overline{f}$ on morphisms. The components $\varepsilon_{(\mathcal{M},A,\bullet)} : (R(\mathcal{M},A,\bullet),*,\mathsf{ev}) \to (\mathcal{M},A,\bullet)$ of the counit $\varepsilon : L \circ R \to 1$ are defined by $\varepsilon_{(\mathcal{M},A,\bullet)}(*) = A$ on objects and by $\varepsilon_{(\mathcal{M},A,\bullet)}(\overline{f}) = f$ on morphisms. That the components of the unit and counit are well-defined as morphisms of the appropriate category is immediate. Next, we show that the unit and counit are natural. For the unit, let $F : \mathcal{M} \to \mathcal{N}$ in $\mathsf{O}_\mathfrak{S}$. Then we have:
  \begin{align*}
    & (\eta_\mathcal{N} \circ F)(f)
      = \eta_\mathcal{N}(F(f))
      = \overline{F(f)}
      = R(F)(\overline{f})
      \\&= R(F)(\eta_\mathcal{M}(f))
      = (R(F) \circ \eta_\mathcal{M})(f)
  \end{align*}
  as required. For the counit, let $F : (\mathcal{M},A,\bullet) \to (\mathcal{N},B,\odot)$ in $\mathsf{T}_\mathfrak{S}$. Then:
  \begin{align*}
    & (\varepsilon_{(N,B,\odot)} \circ (L \circ R)(F))(\overline{f})
      = \varepsilon_{(N,B,\odot)}(L(R(F))(\overline{f}))
      = \varepsilon_{(N,B,\odot)}(R(F)(\overline{f}))
      \\&= \varepsilon_{(N,B,\odot)}(\overline{F(f)})
      = F(f)
      = F(\varepsilon_{(\mathcal{M},A,\bullet)}(\overline{f}))
      = (F \circ \varepsilon_{(\mathcal{M},A,\bullet)})(\overline{f})
  \end{align*}
  and so the unit and counit are natural. We proceed to show that the triangle identities hold. First, we have $(\varepsilon_L \circ L\eta) =  1_L$ as in:
  \begin{align*}
    & (\varepsilon_L \circ L\eta)_\mathcal{M}(f)
      = \varepsilon_{L(\mathcal{M})}(L(\eta_\mathcal{M})(f))
      = \varepsilon_{L(\mathcal{M})}(\eta_\mathcal{M}(f))
      = \varepsilon_{L(\mathcal{M})}(\overline{f})
      = f
  \end{align*}
  Second, we have $(R\varepsilon \circ \eta_R) = 1_R$ as in:
  \begin{align*}
    & (R\varepsilon \circ \eta_R)_{(\mathcal{M},A,\bullet)}(\overline{f})
      = R(\varepsilon_{(\mathcal{M},A,\bullet)})(\eta_{R(\mathcal{M},A,\bullet)}(\overline{f}))
     \\& = R(\varepsilon_{(\mathcal{M},A,\bullet)})(\overline{\overline{f}})
    = \overline{\varepsilon_{(\mathcal{M},A,\bullet)}(\overline{f})}
    = \overline{f}
  \end{align*}
  Thus, we have an adjunction. It remains to show that $\eta : 1 \to R \circ L$ is an isomorphism. It suffices to show that each component $\eta_\mathcal{M} : \mathcal{M} \to R(\mathcal{M},*,\mathsf{ev})$ has an inverse. To that end, define $\eta_\mathcal{M}^{-1} : R(\mathcal{M},*,\mathsf{ev}) \to \mathcal{M}$ by $\eta^{-1}_\mathcal{M}(*) = *$ on objects, and by $\eta^{-1}_\mathcal{M}(\overline{f}) = f$ on morphisms. That this is well-defined as a morphism of $\mathsf{O}_\mathfrak{S}$ is immediate, and we have:
  \begin{mathpar}
    \eta_\mathcal{M}(\eta^{-1}_\mathcal{M}(\overline{f})) = \eta_\mathcal{M}(f) = \overline{f}

    \text{and}

    \eta^{-1}_\mathcal{M}(\eta_\mathcal{M}(f)) = \eta^{-1}_\mathcal{M}(\overline{f}) = f
  \end{mathpar}
  Thus, $\eta$ is a natural isomorphism.
\end{proof}

In particular, an immediate consequence of Theorem~\ref{thm:adjunction} is that the functor $L$ is fully faithful, so that $\mathsf{O}_\mathfrak{S}$ is equivalent to its essential image in $\mathsf{T}_\mathfrak{S}$ under $L$~\cite[Prop. 1.3]{Gabriel1967}. That is, $\mathsf{O}_\mathfrak{S}$ is equivalent to the full subcategory of $\mathsf{T}_\mathfrak{S}$ on the tuples consisting of an $\mathfrak{S}$-combinatory complete applicative system in the context of its own operad of computable morphisms.

The operad of computable morphisms of a given applicative system is canonical in that it is the smallest context that the applicative system in question may sensibly inhabit. Thus, what Theorem~\ref{thm:adjunction} tells us is that weakly-closed operads are combinatory complete applicative systems \emph{standing alone}, from which only this canonical context is recoverable.

We end by remarking explicitly that, modulo this formal difference between applicative systems in context and applicative systems standing alone, our results demonstrate that $\mathfrak{S}$-combinatory complete applicative systems \emph{are the same thing} as weakly-closed $\mathfrak{S}$-operads, characterising the former.

\section{Concluding Remarks}\label{sec:conclusion}

We have introduced a general notion of combinatory completeness parameterised by a faithful cartesian club, and have used it systematically to obtain characterisations of a number of different kinds of applicative system, summarised in Figure~\ref{fig:results-summary}. Our work subsumes the classical characterisation of combinatory algebras as combinatory complete applicative systems, working both in more general settings and with other notions of completeness. We have given a categorical characterisation of combinatory completeness in terms of structured multicategories, and of extensionality in terms of closed multicategories. Moreover, we have exhibited a close correspondence between combinatory complete applicative systems and weakly-closed operads.

We end by discussing a few directions for future work. First, the results of Section~\ref{sec:weakly-closed-operads} recall the work of Hyland on operadic models of the untyped lambda calculus modulo $\beta$-conversion~\cite{Hyland2017}, which are found therein to be the \emph{semi-closed cartesian operads} (in our terminology the semi-closed $\mathbf{Fun}$-operads). It is important to note that Hyland's notion of semi-closed operad is strictly stronger than the notion of weakly-closed operad we use in this paper (see ~\cite[Sect. 1]{Hasegawa2023} for a detailed presentation of semi-closed operads). We conjecture that the relationship between combinatory complete applicative systems and weakly-closed operads can be extended, yielding an analogous relationship between (substructural notions of) $\lambda$-models (see e.g.,~\cite{Hindley2008,Jacobs1993}) and semi-closed operads.

Second, we would like to extend our approach to capture other sorts of applicative system. For example, in this paper the $\mathsf{BWI}$-algebras and $\mathsf{BKWI}$-algebras are conspicuously missing. More broadly we would like to be able to express combinatory completeness of \emph{partial} combinatory algebras, and also more exotic kinds of applicative system, such as the linear combinatory algebras of Abramsky et al.~\cite{Abramsky2002}, the monadic combinatory algebras of Cohen et al.~\cite{Cohen2025}, the braided and ribbon combinatory algebras of Hasegawa et al.~\cite{Hasegawa2021,Hasegawa2024}, and the various sorts of planar combinatory algebra considered by Tomita~\cite{Tomita2021,Tomita2022,Tomita2023}. A related question is for which notions of generalized multicategory (see e.g.,~\cite{Crutwell2010}) the concepts of applicative system and combinatory completeness make sense, which may lead to any number of variations on the theme of combinatory algebra.

\subsection*{Acknowledgements}

We thank the anonymous reviewers of the earlier conference paper that this article extends for identifying a number of mistakes and omissions in the original submission, and thank Matt Earnshaw and Mike Shulman for useful conversations.

\bibliographystyle{fundam}
\bibliography{citations.bib}

\end{document}